\documentclass{amsart}
\usepackage[utf8]{inputenc}

\usepackage{amsmath}
\usepackage{amssymb}
\usepackage{amscd}
\usepackage{amstext}
\usepackage[all]{xy}
\xyoption{2cell}
\UseTwocells
\usepackage{graphicx}
\usepackage{color}
\usepackage[colorlinks=true, allcolors=blue,breaklinks]{hyperref}
\usepackage{natbib}

\usepackage{tikz}
\usepackage{tikz-cd}

\usetikzlibrary{calc, arrows.meta, bending}

\tikzcdset{graphstyle/.append style={row sep=1em, column sep=1em, nodes={inner sep=0.5pt}, arrows={-stealth}}}

\usetikzlibrary{calc, arrows.meta, bending}

\usepackage{mathtools}
\mathtoolsset{showonlyrefs}

\def\al{\alpha}
\def\be{\beta}

\def\de{\delta}

\def\si{\sigma}
\def\Si{\Sigma}

\def\ph{\phi}

\newcommand{\calA}{{\mathcal A}}
\newcommand{\calB}{{\mathcal B}}

\newcommand{\calD}{{\mathcal D}}
\newcommand{\calF}{{\mathcal F}}
\newcommand{\calI}{{\mathcal I}}

\newcommand{\calT}{{\mathcal T}}
\newcommand{\calX}{{\mathcal X}}
\newcommand{\calY}{{\mathcal Y}}
\newcommand{\iso}{\cong}

\newcommand{\bbI}{{\mathbb I}}
\newcommand{\bbN}{{\mathbb N}}

\newcommand{\bbR}{{\mathbb R}}
\newcommand{\bbZ}{{\mathbb Z}}
\newcommand{\bbU}{{\mathbb U}}

\newcommand{\ds}{\oplus}
\newcommand{\Ds}{\bigoplus}

\renewcommand{\mod}{\operatorname{mod}}

\newcommand{\proj}{\operatorname{proj}}
\newcommand{\inj}{\operatorname{inj}}

\newcommand{\Hom}{\operatorname{Hom}}
\newcommand{\RHom}{\operatorname{RHom}}
\newcommand{\RF}{\operatorname{RF}}
\newcommand{\End}{\operatorname{End}}
\newcommand{\Ext}{\operatorname{Ext}}
\newcommand{\Tor}{\operatorname{Tor}}
\newcommand{\rep}{\operatorname{rep}}

\newcommand{\Ker}{\operatorname{Ker}}
\newcommand{\Cok}{\operatorname{Coker}}
\newcommand{\im}{\operatorname{Im}}
\newcommand{\coim}{\operatorname{Coim}}
\newcommand{\blank}{\operatorname{-}}
\newcommand{\op}{\operatorname{op}}

\newcommand{\AR}{\mathrm{\Gamma}}

\def\blank{\operatorname{-}}

\def\id{1\kern-.25em{\text{{\rm l}}}} %1\!\!{\text{{\rm l}}}   is good for 12pt

\newcommand{\Db}{\mathsf{D}^{\mathrm{b}}}
\newcommand{\Kb}{\mathsf{K}^{\mathrm{b}}}
\newcommand{\Cb}{\mathsf{C}^{\mathrm{b}}}
\newcommand{\Supp}{\operatorname{Supp}}

\DeclareMathOperator{\starnd}{\star_{\text{nd}}}

\usepackage{prettyref}

\newcommand{\pref}{\prettyref}
\newrefformat{thm}{Theorem~\ref{#1}}
\newrefformat{cor}{Corollary~\ref{#1}}
\newrefformat{lem}{Lemma~\ref{#1}}
\newrefformat{dfn}{Definition~\ref{#1}}
\newrefformat{cl}{Claim~\ref{#1}}
\newrefformat{sl}{Sublemma~\ref{#1}}
\newrefformat{prp}{Proposition~\ref{#1}}
\newrefformat{conj}{Conjecture~\ref{#1}}
\newrefformat{st}{Step~\ref{#1}}
\newrefformat{sec}{Section~\ref{#1}}
\newrefformat{rem}{Remark~\ref{#1}}
\newrefformat{q}{Question~\ref{#1}}
\newrefformat{pb}{Problem~\ref{#1}}
\newrefformat{cd}{Condition~\ref{#1}}
\newrefformat{eg}{Example~\ref{#1}}
\newrefformat{he}{Heore~\ref{#1}}
\newrefformat{fig}{Figure~\ref{#1}}
\newrefformat{tb}{Table~\ref{#1}}
\newrefformat{as}{Assumption~\ref{#1}}
\newrefformat{subsec}{Subsection~\ref{#1}}
\DeclareMathOperator{\MS}{{\operatorname{SS}}}
\DeclareMathOperator{\Vect}{\operatorname{Vect}}
\newcommand{\Sh}{\operatorname{Sh}}
\newcommand{\simto}{\xrightarrow{\sim}}

\usepackage{cleveref}

\theoremstyle{plain}
\newtheorem{theorem}{Theorem}[section]
\crefname{theorem}{Theorem}{Theorems}
\newtheorem{corollary}[theorem]{Corollary}
\crefname{corollary}{Corollary}{Corollaries}
\newtheorem{lemma}[theorem]{Lemma}
\crefname{lemma}{Lemma}{Lemmas}
\newtheorem{definition-lemma}[theorem]{Definition-Lemma}

\crefname{claim}{Claim}{Claims}

\newtheorem{proposition}[theorem]{Proposition}
\crefname{proposition}{Proposition}{Propositions}

\crefname{conjecture}{Conjecture}{Conjectures}

\crefname{assumption}{Assumption}{Assumptions}

\theoremstyle{definition}
\newtheorem{definition}[theorem]{Definition}
\crefname{definition}{Definition}{Definitions}
\newtheorem{remark}[theorem]{Remark}
\crefname{remark}{Remark}{Remarks}

\newtheorem{example}[theorem]{Example}
\crefname{example}{Example}{Examples}

\crefname{notation}{Notation}{Notations}

\title[DAST for zigzag PM]{Algebraic stability theorem for derived categories of zigzag persistence modules}
\author{Yasuaki Hiraoka \and Yuichi Ike \and Michio Yoshiwaki}

\address[Y.~Hiraoka]{Kyoto University Institute for Advanced Study, WPI-ASHBi, Kyoto University/RIKEN Center for Advanced Intelligence Project}
\email{hiraoka.yasuaki.6z@kyoto-u.ac.jp}
\address[Y.~Ike]{Fujitsu Ltd., Fujitsu Research}
\email{yuichi.ike.1990@gmail.com, ike.yuichi@fujitsu.com}
\address[M.~Yoshiwaki]{RIKEN Center for Advanced Intelligence Project/Osaka City University Advanced Mathematical Institute(OCAMI)}
\email{michio.yoshiwaki@riken.jp, yosiwaki@sci.osaka-cu.ac.jp}

\date{\today}

\keywords{Algebraic stability theorem, Derived category, Auslander--Reiten quiver, Persistence module, Zigzag persistence module, Constuctible sheaf}
\subjclass[2010]{16G20, 16E35, 55N99, 35A27}

\begin{document}

\maketitle

\begin{abstract}
We study distances on zigzag persistence modules from the viewpoint of derived categories and Auslander--Reiten quivers. 
The derived category of ordinary persistence modules is derived equivalent to that of arbitrary zigzag persistence modules, depending on a classical tilting module.
Through this derived equivalence, we define and compute distances on the derived category of arbitrary zigzag persistence modules and prove an algebraic stability theorem.
We also compare our distance with the distance for purely zigzag persistence modules introduced by Botnan--Lesnick and the sheaf-theoretic convolution distance due to Kashiwara--Schapira.
\end{abstract}

\section{Introduction}

Topological data analysis has recently become popular for studying the shape of data in various research areas (\citealt{hiraoka2016hierarchical,lee2017quantifying,oyama2019hepatic,saadatfar2017pore}). 
In topological data analysis, one of the standard tools is persistent homology, the original concept for which was introduced by \citet*{edelsbrunner2002topological}.
For a filtration 
\begin{equation}
    \mathbb{X}:X_1 \hookrightarrow X_2 \hookrightarrow \cdots \hookrightarrow X_n
\end{equation}
of topological spaces, the \emph{$q$-th persistent homology} is defined by
\begin{equation}
    H_q (\mathbb{X}) : H_q (X_1) \to H_q (X_2) \to \cdots \to H_q (X_n),
\end{equation}
where $H_q (\blank)$ is the $q$-th homology functor with a field coefficient.
Persistent homology is utilized to study the persistence of topological features in the filtration $\mathbb{X}$ such as connected components, loops, voids, and so on, for each dimension $q$. 
The algebraic structure of persistent homology is expressed by the notion of \emph{(1D) persistence modules}, which are representations of an equioriented $A_n$-type quiver.
This was pointed out by \citet{carlsson2010zigzag}.

Representations of a non-equioriented $A_n$ are also studied from both of the theoretical and practical perspectives in topological data analysis. 
Here, a representation of an $A_n$-type quiver with alternating (resp.\ arbitrary) orientation is called a \emph{purely zigzag (resp.\ zigzag) persistence module}, while 1D persistence modules are said to be \emph{ordinary}.
Zigzag persistence modules can also be applied to address characteristic topological features not captured by the theory of ordinary persistence modules (\citealt{carlsson2010zigzag}). 
For example, let us study time-series data given by a sequence $X_1, \cdots, X_t, \cdots ,X_T$ of topological spaces $X_t$ for each time $t$. 
In general, this sequence is not a filtration with respect to $t$, but we can consider the following zigzag diagram:
\begin{equation}
    X_1 \hookrightarrow X_1 \cup X_2 \hookleftarrow X_2 \hookrightarrow \cdots \hookrightarrow X_{T-1} \cup X_T \hookleftarrow X_T.
\end{equation}
By applying a homology functor $H_q (\blank)$ to this diagram, we obtain a purely zigzag persistence module
\begin{equation}
    H_q(X_1) \to H_q(X_1 \cup X_2) \leftarrow H_q(X_2) \to \cdots \to H_q(X_{T-1} \cup X_T) \leftarrow H_q(X_T).
\end{equation}

From \citet{gabriel1972unzerlegbare} and the Krull-Schmidt Theorem, any representation of an $A_n$-quiver can be uniquely decomposed into \emph{interval representations}, which are exactly indecomposable representations in this setting.
The endpoints of these interval representations define the birth-death parameters of the topological features, and those topological features are summarized in a \emph{barcode} (or a \emph{persistence diagram}). 
Recently, the notion of barcodes with grading was introduced to capture topological features for all degree in persistent homology. 
A barcode with grading is defined for an object in the derived category of persistence modules, namely, a cochain complex of persistence modules, and has been studied by \citet{berkouk2018derived}. 
Since there is a fully faithful embedding of the category of persistence modules into its derived category, the usual barcode of a persistence module is a special case of a barcode with grading.
This generalization enables us to discover some interactions between different degrees in persistent homology. 
We cannot capture these interactions while considering only one degree.

The Krull-Schmidt Theorem reduces the description of the category of representations of quivers into that of the full subcategory consisting of indecomposable representations.
To explicitly compute indecomposable representations, 
the Auslander-Reiten (AR) quiver was introduced (see \citealt{auslander1997representation}) and has been studied in representation theory of finite-dimensional algebras since the 1970s.
For details on the AR quiver, refer to \citet[Section 3.1]{schiffler2014quiver} (see \citealt[Chapter IV]{assem2006elements} for a more general setting).
From the viewpoint of AR theory, the barcode of a persistence module can be defined as a map from the set of vertices in the AR quiver of the equioriented $A_n$-type quiver to the integers, sending an interval representation to its multiplicity. 
Moreover, the AR quiver of the derived category of persistence modules can be defined. 
Then the \emph{barcode} of an object of the derived category is canonically defined as a map from the set of vertices in the AR quiver to the integers. 

For persistent homology, it is significant that a stability theorem holds, which says that the barcode is stable (precisely, 1-Lipschitz) with respect to small changes in a given dataset.
This stability theorem was first proved by \citet*{cohen2007stability} for the persistent homology of the sublevel set filtration induced by an $\bbR$-valued function. 
The algebraic perspective on persistent homology allows for a generalization of the stability theorem, the so-called algebraic stability theorem (AST). 
\citet{chazal2009proximity} introduced the interleaving distance between ordinary persistence modules to weaken the assumptions needed for the stability theorem, and then proved the AST with that distance, which states that the barcode is stable with respect to the interleaving distance. 
Following this algebraic generalization, \citet{bauer2015induced} provided a simpler proof of the AST via the induced matching theorem (IMT) (see Theorem~\ref{thm:IMT}).
Note that the converse of the AST also holds (\citealt{lesnick2015theory}), hence giving the isometry theorem for persistence modules.

It is natural to expect that an AST holds for  purely zigzag persistence modules, and it was indeed obtained in \citet{botnan2018algebraic}.
For this purpose, they introduced a distance on purely zigzag persistence modules, which we call the \emph{block distance} following \citet{meehan2020persistence}.
\citet{bjerkevik2016stability} improved the theorem with a tight bound and provided an isometry theorem for purely zigzag persistence modules.
Note that zigzag persistence modules in \citet{botnan2018algebraic} 
and
\cite{bjerkevik2016stability} are purely zigzag ones in our convention.

In this paper, we prove an AST for the derived category of zigzag persistence modules through a correspondence between the AR quivers of ordinary and zigzag persistence modules. 
The distance between interval representations can be understood through the AR quiver, thus that for the derived category should be understood in the same way.
For our strategy, we consider a derived equivalence, that is, an equivalence as triangulated categories between the derived categories, which induces an isomorphism between their AR quivers. 
In particular, we treat the derived equivalence between the categories of ordinary and zigzag persistence modules associated with a classical tilting modules (see \citealt{happel1988triangulated}) since we can compute the induced distance from the equivalence explicitly. 
Our strategy enables us to obtain an AST for the wider class (i.e., arbitrary orientations) in a unified manner than the result of \citet{botnan2018algebraic}. 
Moreover, our approach gives clear understanding of distances on the derived category of zigzag persistence modules through AR quivers.

We first define two distances on the derived category of ordinary persistence modules and generalize the AST in the derived setting.
For a derived category $\calD$ equivalent to the derived category of ordinary persistence modules, two distances on $\calD$ can be canonically obtained from the ones on the latter derived category. 
In fact, the derived equivalences preserve isomorphisms, the indecomposability of objects, and the AR quivers, which can be utilized to define and compute the distances.
The derived categories of ordinary and zigzag persistence modules are derived equivalent associated with a classical tilting module (see \citealt{assem2006elements,brenner1980generalizations,bongartz1981tilted,happel1982tilted}).
Indeed, for a classical tilting module $T$, the right derived functor of $\Hom (T, \blank)$ gives a derived equivalence (see \citealt{happel1988triangulated}).
This derived equivalence induces the distances on the derived category of zigzag persistence modules and provides an AST for the derived category of zigzag persistence modules.
More strongly, we also obtain an isometry theorem in this setting.

By construction, our distance can be computed explicitly through the correspondence of the AR quivers of the derived categories of zigzag and ordinary persistence modules.
Based on the calculation, we compare our induced distance with other two distances.
The first one is the block distance introduced by \citet{botnan2018algebraic}.
%As mentioned above, an AST for this distance was first proved by \citet{botnan2018algebraic} and improved by \citet{bjerkevik2016stability} as an isometry theorem. 
We will show that the block distance utilized in the above theorem is incomparable with ours.
The second one is the \emph{convolution distance} on the derived category of constructible sheaves due to \citet{kashiwara2018persistent}.
\citet{berkouk2018derived} proved the derived isometry theorem for the convolution distance.
Indeed, they introduced a bottleneck-type distance on graded barcodes and then proved that this coincides with the convolution distance. 
To compare our distance with the convolution distance, we show that a non-derived version of the convolution distance is equal to the block distance in \cite{botnan2018algebraic}.
For the proof, we investigate (i) the relation between representations of the infinite purely zigzag quiver and constructible sheaves on $\bbR$, and (ii) the relation between 2D persistence modules and constructibe sheaves on $\bbR$, which was studied by \citet{berkouk2019level}.
As a result, we confirm that the convolution distance is also incomparable with our induced distance.

The remainder of the paper is organized as follows. 
Section~\ref{sec:prelim} reviews the basic concepts of persistent homology from the viewpoint of representation theory and recalls the AST in the ordinary setting. 
Section~\ref{sec:dercat} reviews the basics of derived categories and studies the derived category of persistence modules.
Then we introduce the interleaving distance and the bottleneck distance in the derived setting.
Section~\ref{sec:main} proves the main results: the AST for the derived category of ordinary persistence modules and that for zigzag persistence modules.
Section~\ref{sec:DIST} extends the result of Section~\ref{sec:main} to isometry theorems.
In Section~\ref{sec:Dircal}, we explicitly calculate our induced distance on zigzag persistence modules. 
Section~\ref{sec:comparison} confirms that our induced distance and the block distance  are incomparable in the purely zigzag setting.
Finally, Section~\ref{sec:compsheaf} also verifies that the convolution distance is incomparable with ours via the block distance. 

\section{Preliminaries} \label{sec:prelim}
\subsection{Quiver representations}
\label{subsec:quiver_representation}
Throughout this paper, $\Bbbk$ denotes an algebraically closed field, and
all vector spaces, algebras, and linear maps are assumed to be finite-dimensional
$\Bbbk$-vector spaces, finite-dimensional $\Bbbk$-algebras, and $\Bbbk$-linear maps, respectively.
Furthermore, all categories and functors are assumed to be additive.

A \emph{quiver} $Q$ is a directed graph.
Formally, a quiver $Q$ is a quadruple $Q=(Q_0,Q_1,s,t)$ of 
sets $Q_0$ of vertices and $Q_1$ of arrows, and maps $s,t \colon Q_1 \to Q_0$.
We draw an arrow $\alpha\in Q_1$ as $\alpha \colon 1\to 2$ if $s(\alpha)=1, t(\alpha)=2\in Q_0$.
The \emph{opposite quiver} $Q^{\op}$ of a quiver $Q=(Q_0,Q_1,s,t)$ is $Q^{\op}=(Q_0,Q_1,t,s)$. 
For example, the opposite quiver of $1\to 2$ is $1 \leftarrow 2$.
A quiver $Q$ is \emph{finite} if $Q_0$ and $Q_1$ are finite. 

A \emph{quiver morphism} $f$ from a quiver $Q=(Q_0,Q_1,s,t)$ to a quiver $Q'=(Q'_0,Q'_1,s',t')$ is a pair $f=(f_0,f_1)$ of maps $f_i \colon Q_i\to Q'_i$ for $i=0,1$ such that $s'\circ f_1 = f_0 \circ s$ and $t'\circ f_1 = f_0 \circ t$. 
For example, $\id_Q = (\id_{Q_0},\id_{Q_1})$ is a quiver morphism $Q \to Q$, which is called the \emph{identity morphism}. 
A quiver morphism $f \colon Q\to Q'$ is called an \emph{isomorphism} if there is a quiver morphism $g \colon Q' \to Q$ such that $f\circ g = \id_{Q'}$ and $g \circ f = \id_Q$.
A quiver $Q$ is \emph{isomorphic} to a quiver $Q'$, denoted by $Q \cong Q'$, if there is an isomorphism from $Q$ to $Q'$.
For example, a quiver of the form $1 \xrightarrow{\al} 2$ is isomorphic to a quiver of the form $X \xrightarrow{\be} Y$. Indeed, we have an isomorphism $f=(f_0,f_1)$ defined by $f_0(1)=X,f_0(2)=Y$, and $f_1(\al)=\be$.

Here, we introduce the $A_n$-type quiver $A_n(a)$ with orientation $a$, whose underlying graph is the Dynkin diagram of type $A:1$ \textbf{---} $2$ \textbf{---} $\cdots$ \textbf{---} $n$ for $n\in \bbN$.
Then $A_n(a)$ is the quiver 
\begin{equation}
    1 \leftrightarrow 2 \leftrightarrow \cdots \leftrightarrow n,
\end{equation}
where $\leftrightarrow$ means  $\to$ or $\leftarrow$ assigned by the orientation $a$.
In this paper, the following $A_n$-type quivers with certain orientations are frequently used.
The $A_n$-type quiver with equi-orientation 
\begin{equation}
    1\to 2 \to \cdots \to n
\end{equation}
is called the \emph{equioriented $A_n$-type quiver}, 
which is denoted by $A_n(=A_n(e))$. 
The $A_n$-type quiver with alternating orientation is called a \emph{purely zigzag $A_n$-type quiver}, which is denoted by $A_n(z)$.
Moreover, if the vertex $1$ of a purely zigzag $A_n$-type quiver $Q$ is a sink vertex, $Q$ is denoted by $A_n (z_1)$.
Otherwise, it is denoted by $A_n (z_2)$.
Namely, $A_n(z_1)$ is the following quiver:
\begin{equation}
    1\leftarrow 2 \to 3 \leftarrow \cdots \to n \text{ if $n$ is odd, }
1\leftarrow 2 \to 3 \leftarrow \cdots \leftarrow n \text{ if $n$ is even, } 
\end{equation}
and $A_n(z_2)$ is the following quiver:
\begin{equation}
    1\to 2 \leftarrow 3 \to \cdots \leftarrow n \text{ if $n$ is odd, }
1\to 2 \leftarrow 3 \to \cdots \to n \text{ if $n$ is even. } 
\end{equation}

A \emph{representation} $M$ of a quiver $Q$ is a family of vector spaces $M_x$ at each vertex $x \in Q_0$ and linear maps $M_\al$ on each arrow $\al\in Q_1$.
For example, a representation $M$ of the equioriented $A_n$-type quiver 
\begin{equation}
    A_n: 1\xrightarrow{\alpha_{1,2}} 2 \xrightarrow{\alpha_{2,3}} \cdots \xrightarrow{\alpha_{n-1,n}} n
\end{equation}
has the following form:
\begin{equation}
    M_1\xrightarrow{M_{\alpha_{1,2}}} M_2 \xrightarrow{M_{\alpha_{2,3}}} \cdots \xrightarrow{M_{\alpha_{n-1,n}}} M_n.
\end{equation}
A \emph{subrepresentation} $N$ of $M$ is defined as a representation of $Q$ such that $N_x \subseteq M_x$ for each vertex $x\in Q_0$ and $N_{\al}=M_{\al}|_{N_x}$ for each arrow $\al\colon x\to y\in Q_1$. 
The \emph{direct sum} $M \oplus N$ of representations $M$ and $N$ is defined by $(M \oplus N)_x = M_x \oplus N_x$ for each vertex $x\in Q_0$ and $(M \oplus N)_\al = M_\al \oplus N_\al$ for each arrow $\al\in Q_1$. 
The \emph{dimension} of  $M$ is defined by $\dim M:= \sum_{x\in Q_0} \dim M_x $.
All representations $M$ are assumed to be \emph{pointwise finite-dimensional}, namely $\dim M_x < \infty$ for each $x\in Q_0$. 
When $Q$ is finite, this is just \emph{finite-dimensional}, that is, $\dim M < \infty$. 

Let $M,N$ be representations of $Q$.
Then a \emph{morphism} $f \colon M \to N$ is a family of linear maps $f_x \colon M_x \to N_x$ on each vertex $x\in Q_0$ such that the following diagram commutes for any arrow $\alpha:x \to y \in Q_1$:
\begin{equation}
    \xymatrix{
M_x \ar[r]^{f_x} \ar[d]_{M_{\alpha}} & N_x \ar[d]^{N_{\alpha}}\\
M_y \ar[r]_{f_y} & N_y. \\
}
\end{equation}
For example, $\id_M = (\id_{M_x})_{x\in Q_0}$ is a morphism $M \to M$, which is called the \emph{identity morphism}. 
A morphism $f \colon M\to N$ is called an \emph{isomorphism} if there is a morphism $g \colon N \to M$ such that $f\circ g = \id_N$ and $g \circ f = \id_M$.
A representation $M$ is \emph{isomorphic} to a representation $N$, denoted by $M \cong N$, if there is an isomorphism from $M$ to $N$.
Moreover, a non-zero representation $M$ is said to be \emph{indecomposable} if 
$M \cong N \oplus L$ implies $N=0$ or $L=0$. 

The abelian category of representations of $Q$ is denoted by $\rep_{\Bbbk} Q$. 
Note that $\rep_{\Bbbk} Q$ is a Krull-Schmidt category when $Q$ is finite (see \citealt[p.11, Theorem 1.2]{schiffler2014quiver} for example).
Indeed, 
for any $M \in \rep_{\Bbbk} Q$, we have unique decomposition 
\begin{equation}
    M \cong M^1 \oplus \cdots \oplus M^s
\end{equation}
up to permutations and isomorphisms,
where each $M^i$ is indecomposable.
More generally, when $Q$ is the infinite zigzag quiver (see Section~\ref{subsec:bldist}), any $M\in\rep_{\Bbbk} Q$ has unique indecomposable (infinite) decomposition up to permutations and isomorphisms by Krull-Schmidt-Remak-Azumaya Theorem since every indecomposable representation of $Q$ is an interval one (see \citealt{botnan2017interval}).

\subsection{Persistence modules}
\label{subsec:per}

We call each $M\in \rep_{\Bbbk} A_n$, each $N\in \rep_{\Bbbk} A_n(z)$, and each $L\in \rep_{\Bbbk} A_n(a)$  a \emph{(ordinary) persistence module},  a \emph{purely zigzag persistence module}, and  a \emph{zigzag persistence module}, respectively.
In this subsection, we will define the internal morphisms of an ordinary persistence module and an endofunctor of the category of ordinary persistence modules in order to define the interleaving distance.

For any $A_n$-type quiver $A_n(a)$, $\alpha_{x,y}$ denotes the arrow between $x$ and $y$ with $1\leq x < y \leq n$. 
Then the equioriented $A_n$-type quiver $A_n$ is 
\begin{equation}
    A_n: 1\xrightarrow{\ \alpha_{1,2}\ } 2 \xrightarrow{\ \alpha_{2,3}\ } \cdots \xrightarrow{\alpha_{n-1,n}} n
\end{equation}
and a persistence module $M$ has the form
\begin{equation}
M_1 \xrightarrow{\ M_{\alpha_{1,2}}\ } M_2 \xrightarrow{\ M_{\alpha_{2,3}}\ } \cdots \xrightarrow{M_{\alpha_{n-1,n}}} M_n.
\end{equation}
Moreover, when $n$ is odd, 
the purely zigzag $A_n$-type quiver $A_n(z_1)$ is
\begin{equation}
1\xleftarrow{\ \alpha_{1,2}\ } 2 \xrightarrow{\ \alpha_{2,3}\ } \cdots \xrightarrow{\alpha_{n-1,n}} n
\end{equation} 
and a purely zigzag persistence module $M\in \rep_{\Bbbk} A_n(z_1)$ has the form 
\begin{equation}
M_1 \xleftarrow{\ M_{\alpha_{1,2}}\ } M_2 \xrightarrow{\ M_{\alpha_{2,3}}\ } \cdots \xrightarrow{M_{\alpha_{n-1,n}}} M_n. 
\end{equation}
In other cases, we can similarly express the zigzag $A_n$-type quivers and the zigzag persistence modules.  

\begin{definition} \label{dfn:per}
Let $M,N$ be persistence modules and $\de$ an integer.\\
(1) For $1\leq s\leq t \leq n$,
the linear map $\phi_M (s,t) \colon M_s\to M_t$ is defined by
\begin{equation}
\phi_M (s,t)
=\left\{
\begin{array}{cc}
\id_{M_s}, & s=t  \\
M_{\alpha_{t-1,t}}\circ \cdots \circ M_{\alpha_{s,s+1}}, & \text{otherwise}
\end{array}.
\right.
\end{equation}
By definition, we have $\ph_M(s,t)=\ph_M (t-1,t) \circ \cdots \circ \ph_M (s,s+1)$. \\
(2) 
The $\de$-shift $M(\de)$ of $M$ is defined by  
\begin{equation}
(M(\de))_x= \left\{ \begin{array}{ll}
M_{x+\de}, & 1\leq x+\de \leq n \\
0, & \text{otherwise}
\end{array}
\right.
\end{equation}
and
\begin{equation}
(M(\de))_{\alpha_{x,x+1}}= \left\{ \begin{array}{ll}
M_{\alpha_{x+\de,x+1+\de}}, & 1\leq x+\de \leq x+1+\de \leq n \\
0, & \text{otherwise}
\end{array}
\right.
\end{equation}
for each vertex $x$ of $A_n$.
For a morphism $f \colon M\to N$ in $\rep_{\Bbbk} A_n$, 
the $\de$-shift $f(\de)$ of $f$ is defined by 
\begin{equation}
(f(\de))_x= \left\{ \begin{array}{ll}
f_{x+\de}, & 1\leq x+\de \leq n \\
0, & \text{otherwise}
\end{array}
\right.
\end{equation}
for each vertex $x$ of $A_n$.
This defines the \emph{$\de$-shift functor} $(\de) \colon \rep_{\Bbbk} A_n \to \rep_{\Bbbk} A_n$. It should be noted that the $\delta$-shift functor can only be defined in the equioriented setting.\\
(3)
Assume that $\delta$ is non-negative.
The \emph{transition morphism} $\ph_M^\de \colon M \to M (\de)$ in $\rep_{\Bbbk} A_n$ is defined by $(\ph_M^\de)_x=\ph_M (x, x+\de)$ for each vertex $x$ of $A_n$.
For any morphism $f \colon M \to N$, 
we have the following commutative diagram:
\begin{equation}
\xymatrix{
M \ar[r]^{\ph_M^\de} \ar[d]_f & M(\de)\ \ar[d]^{f(\de)} \\
N \ar[r]_{\ph_N^\de} & N(\de).
}
\end{equation}
This defines a natural transformation $\ph^{\de} \colon \id \to (\de)$ from the identity functor $\id$ to the $\de$-shift functor $(\de)$. 
\\
(4)
A persistence module $M$ is \emph{$\de$-trivial} if the transition morphism $\ph_M^\de \colon M \to M (\de)$ is zero.
\end{definition}

In our setting, the functor $(\de)$ is not an equivalence but an exact functor.
Indeed, let $M,N,L$ be persistence modules.
A sequence 
\begin{equation}0 \to M \to N \to L \to 0\end{equation}
is exact if and only if 
the sequence 
\begin{equation}0 \to M_x \to N_x \to L_x \to 0\end{equation} 
is exact for each vertex $x$ of $A_n$.
This means that the sequence 
\begin{equation}0 \to M(\de) \to N(\de) \to L(\de) \to 0\end{equation} 
is exact. 

\subsection{Interleaving distance} \label{subsec:intl}

In this paper, a distance on a set $X$ means an extended pseudometric.  
Specifically, it is a function $d \colon X \times X \to \bbR_{\geq 0} \cup \{\infty \}$ such that, for every $x,y,z\in X$,
\begin{enumerate}
\item $d(x,x)=0$,
\item $d(x,y)=d(y,x)$, \text{and}
\item $d(x,z) \leq d(x,y) + d(y,z)$ if $d(x,y), d(y,z)<\infty$.
\end{enumerate}

Let us recall the interleaving distance between persistence modules.

\begin{definition} \label{dfn:interleave}
Let $\de$ be a non-negative integer.
Two persistence modules $M$ and $N$ are said to be \emph{$\de$-interleaved} 
if there exist morphisms $f \colon M \to N(\de)$ and $g \colon N \to M(\de)$ such that the following diagrams commute:
\begin{equation}
\xymatrix{
M \ar[rr]^{\ph_M^{2\de}} \ar[dr]_f && M(2\de), & N \ar[rr]^{\ph_N^{2\de}} \ar[dr]_g && N(2\de).\\
& N(\de) \ar[ur]_{g(\de)} & & & M(\de) \ar[ur]_{f(\de)}&
}
\end{equation}
In this case, we call the pair of $f\colon M \to N(\de)$ and $g \colon N \to M (\de)$ a \emph{$\de$-interleaving pair}. 
Moreover, we call a morphism $f \colon M \to N(\de)$ a \emph{$\de$-interleaving morphism} if there is a morphism $g \colon N \to M(\de)$ such that the pair $(f,g)$ is a $\de$-interleaving pair.

For persistence modules $M,N$,
the \emph{interleaving distance} is defined as
\begin{equation}d_I(M,N) := \inf\{ \de \in \bbZ_{\geq 0} \mid M \text{ and } N \text{ are $\de$-interleaved} \}.\end{equation}

\end{definition}

We remark that in our setting, $d_I(M,N)=0$ if and only if $M$ and $N$ are isomorphic.
Thus, the interleaving distance measures how far these modules are from being isomorphic.
We will extend this concept to the derived setting later (see Definition~\ref{dfn:dinterleave}).

\subsection{Intervals and barcodes} \label{subsec:intbar}

We recall that the category $\rep_{\Bbbk} A_n (a)$ of zigzag persistence modules is a Krull-Schmidt category,
i.e., a representation of $A_n (a)$ is isomorphic to a direct sum of indecomposable representations.
In this subsection, we discuss all indecomposable representations of $A_n(a)$. 

\begin{definition}
For $1\leq b\leq d \leq n$, the \emph{interval representation} $\bbI[b,d] \in \rep_{\Bbbk} A_n(a)$ is defined by
\begin{equation}
(\bbI [b,d])_x :=
\left\{ \begin{array}{ll}
\Bbbk, & b\leq x \leq d \\
0, & \text{otherwise}
\end{array}
\right.
\end{equation}
and 
\begin{equation}
(\bbI [b,d])_{\alpha_{x,y}}:=
\left\{ \begin{array}{ll}
\id_\Bbbk, & b\leq x < y \leq d \\
0, & \text{otherwise}
\end{array}.
\right.
\end{equation}
\end{definition}

Any interval representation is indecomposable.
The converse also holds as follows.

\begin{theorem}[\citealt{gabriel1972unzerlegbare}] \label{thm:gab}
Any indecomposable representation of $A_n(a)$ is isomorphic to an interval representation $\bbI[b,d]$ for some $1\leq b \leq d \leq n$. 
\end{theorem}

Thus, for a representation $M$ of $A_n(a)$, we obtain the unique interval decomposition
\begin{equation}
M \iso \Ds_{1\leq b\leq d \leq n} \bbI[b,d]^{m(b,d)},
\end{equation}
leading to the definition of the \emph{barcode} (or the \emph{persistence diagram}) $\calB (M)$ of $M$ by 
\begin{equation}\{ (b,d,m) \mid 1\leq b\leq d \leq n, 1\leq m \leq m(b,d) \text{ such that } m(b,d)\not = 0 \}.\end{equation}

Below, we will use the same notation as in \citet{bauer2015induced}.
Recall that a multiset is a pair $(S,m)$ of a set $S$ and a map $m\colon S\to \bbZ_{>0}$.
A representation $Rep(S,m)$ of a multiset $(S,m)$ is  
\begin{equation}
Rep(S,m)=\{(s,i)\in S\times \bbZ_{>0} \mid i \leq m(s) \}.
\end{equation}
If $S$ is a totally ordered set, then a representation $Rep(S,m)$ is also a totally ordered set with order obtained by restricting the lexicographic order on $S\times \bbZ_{> 0}$ to $Rep(S,m)$.

For a persistence module $M$, the barcode $\calB (M)$ is regarded as a representation of a multiset $(\calI_M,m)$ of the set $\calI_M:=\{ (b,d) \mid m(b,d)\not =0 \}$ and the map $m\colon \calI_M \to \bbZ_{>0}$ sending $(b,d)$ to $m(b,d)$. 
For simplicity, write an element $(b,d,m)$ of $\calB (M)$ as $\langle b, d\rangle$, which is called an \emph{interval}. 
For $1\leq b \leq n $, $\calB (M)_{\langle b, \blank  \rangle}$ denotes the subset of $\calB (M)$ consisting of the intervals $\langle b,c  \rangle$ for some $b\leq c\leq n$, 
and $\calB (M)_{\langle \blank, d  \rangle}$ denotes the subset of $\calB (M)$ consisting of the intervals $\langle c, d  \rangle$ for some $1\leq c\leq d$.  
Note that $\calB (M)_{\langle b, \blank  \rangle}$ and $\calB (M)_{\langle \blank, d  \rangle}$ are regarded as totally ordered sets with the total order induced by the reverse inclusion relation on intervals.
Indeed, if $c< c'$, then $\langle b,c' \rangle <\langle b,c \rangle$ in $\calB (M)_{\langle b, \blank \rangle}$ and $\langle c,d \rangle <\langle c',d \rangle$ in $\calB (M)_{\langle \blank, d  \rangle}$.

From the perspective of AR theory, the barcode of a representation $M$ of $A_n(a)$ can be defined as a map $\AR_0 \to \bbZ$ sending an interval $\bbI [b,d]$ to its multiplicity $m(b,d)$ in the decomposition of $M$, where $\AR_0$ is the set of all interval representations.
Note that $\AR_0$ is the set of vertices of the AR quiver of $A_n(a)$ (for details, see \citealt[Section 1.5 and 3.1, Chapter 7]{schiffler2014quiver}), and in this sense, AR quivers are hidden behind barcodes.
The AR quivers are important tools in representation theory of quivers.
Indeed, under a certain assumption, the AR quiver can recover the category of representations. 

\begin{example} \label{ex:ARquiver}
The AR quiver $\AR (A_3)$ of $A_3$ is
\begin{equation}
\AR (A_3)=
    \begin{tikzcd}[graphstyle,every matrix/.append style={name=m}]
      && \bbI [1,3] \arrow[dr] \\
      & \bbI [2,3] \arrow[ur] \arrow[dr] & 
      & \bbI [1,2] \arrow[dr]  \\
      \bbI [3,3] \arrow[ur] &   &\bbI [2,2] \arrow[ur] &  & \bbI [1,1] 
    \end{tikzcd},
\end{equation}
while
the AR quiver $\AR (A_3(z_1))$ of $A_3(z_1):1\leftarrow 2 \to 3$  is 
\begin{equation}
\AR (A_3(z_1))=
    \begin{tikzcd}[graphstyle,every matrix/.append style={name=m}]
     \bbI [1,1] \arrow[dr] &  & \bbI [2,3] \arrow[dr] & \\
      & \bbI [1,3] \arrow[ur] \arrow[dr] &  & \bbI [2,2]  \\
     \bbI [3,3] \arrow[ur]&    &\bbI [1,2] \arrow[ur] &    
    \end{tikzcd}.
\end{equation}
\end{example}

\subsection{Matching and the bottleneck distance} \label{subsec:bottle}
A \emph{matching} from a set $S$ to a set $T$ (written as $\si \colon S \nrightarrow T$) is a bijection $\si \colon S^{\prime} \to T^{\prime} $ for some subset $S^{\prime}$ of $S$ and some subset $T^{\prime}$ of $T$.
For such a matching $\sigma$, we write $S^{\prime}$ as $\coim\si$ and $T^{\prime}$ as $\im \si$.

For totally ordered sets, a matching can be defined canonically as follows:
let $S=\{S_i \mid i = 1,\cdots, s \}$ and $T=\{T_i \mid i = 1,\cdots, t \}$ be finite totally ordered sets such that for $a\leq b$, $S_a \leq S_b$ and $T_a \leq T_b$.
Then a \emph{canonical matching} $\si \colon S \nrightarrow T$ is a matching $\si$ given by $\si (S_i)=T_i$ for $i=1,\cdots, \min \{ s,t \}$.
In this case, either $\im \si = S$ or $\coim \si = T$ is satisfied.

We next define a $\de$-matching between barcodes.

\begin{definition} \label{dfn:bottleneck}
Let $\de$ be a non-negative integer.
For a barcode $\calB$,
let $\calB_\de$ be the subset of $\calB$ consisting of intervals $\langle b,d\rangle$ such that $d-b\geq \de$. 
A \emph{$\de$-matching} between barcodes $\calB$ and $\calB^{\prime}$ is defined by a matching $\si \colon \calB \nrightarrow \calB^{\prime}$ such that
\begin{equation}
\calB_{2\de} \subseteq \coim \si, \quad  \calB^{\prime}_{2\de} \subseteq \im \si  
\end{equation}
and for all $\si \langle b, d \rangle = \langle b^{\prime}, d^{\prime} \rangle$,  
\begin{equation} \label{eq:corrint}
b^{\prime}-\de \leq b \leq d \leq d^{\prime}+\de, \quad b-\de \leq b^{\prime} \leq d^{\prime} \leq d+\de.
\end{equation}
Two barcodes $\calB$ and $\calB^{\prime}$ are said to be \emph{$\de$-matched} if there is a $\de$-matching between $\calB$ and $\calB^{\prime}$.
The \emph{bottleneck distance} between $\calB$ and $\calB'$ is defined as 
\begin{equation}
d_B (\calB,\calB^{\prime}):= \inf \{ \de \in \bbZ_{\geq 0} \mid \text{$\calB$ and $\calB^{\prime}$ are $\de$-matched} \}.
\end{equation}
\end{definition}

Note that equation~\eqref{eq:corrint} implies that the interval representations associated with $\langle b,d \rangle$, $\langle b',d' \rangle$ are $\de$-interleaved.

\subsection{Algebraic stability theorem} \label{subsec:AST}

In this subsection, we will explain the proof of an AST for $\rep_{\Bbbk} A_n$ following \citet{bauer2015induced}.
Their strategy utilizes an IMT.

\begin{definition}
Let $f \colon M \to N$ be a morphism in $\rep_{\Bbbk} A_n$.
Then the \emph{induced matching} $\calB(f) \colon \calB(M) \to \calB(N)$ is defined as follows:
\begin{enumerate}
\item When $f$ is injective, $\calB(f)$ is defined via the family of canonical matchings from $\calB(M)_{\langle \blank,d  \rangle}$ to $\calB(N)_{\langle \blank,d  \rangle}$.
\item When $f$ is surjective, $\calB(f)$ is defined via the family of canonical matchings from $\calB(M)_{\langle b, \blank \rangle}$ to $\calB(N)_{\langle b, \blank  \rangle}$.
\item Any morphism $f$ can be decomposed into the surjective morphism $\pi \colon M \to \im f$ and the injective morphism $\mu \colon \im f \to N$. Then $\calB(f):= \calB(\mu) \circ \calB(\pi)$ by (1) and (2). 
\end{enumerate}
\end{definition}

This matching is what yields the IMT (see \citealt[Theorem 4.2]{bauer2015induced}).
To state the IMT in our setting, we extend representations $M$ in $\rep_{\Bbbk} A_n$ to those in $\rep_{\Bbbk} A_\ell$ for $\ell \geq n$ as
\begin{equation}
0 \to \cdots \to 0 \to M_1 \to \cdots \to M_n \to 0 \to \cdots \to 0 \in \rep_{\Bbbk} A_{\ell}.
\end{equation}
Moreover, for a given representation $M \in \rep_{\Bbbk} A_n$ and non-negative integer $\de$, 
the map $r_M^{\de} \colon \calB(M(\de)) \to \calB(M)$ is defined by $r_M^{\de} \langle b,d \rangle := \langle b+\de,d+\de \rangle$.
In general, the map $r_M^{\de}$ is not bijective. 
However, we can take an integer $\ell\geq n$ large enough such that $M$ and $(M(\de))(-\de)$ are isomorphic as representations of $A_{\ell}$. 
In this case, the map $r_M^{\de}$ is bijective.
Then, the IMT is stated as follows:
\begin{theorem}[IMT] \label{thm:IMT}
Let $f \colon M \to N$ be a morphism in $\rep_{\Bbbk} A_n$. 
Assume that $\Ker f$ and $\Cok f$ are $2\de$-trivial.
Moreover, taking an integer $\ell\geq n$ large enough such that $r_M^\de$ is bijective, one regards $M,N$ as representations of $A_{\ell}$.
Then 
$\calB(f)\circ r_M ^\de$ is a $\delta$-matching $\calB(M(\de)) \nrightarrow \calB(N)$.
\end{theorem}

Let $f \colon M \to N(\de)$ be a $\de$-interleaving morphism.
It is easily seen that $\Ker f$ and $\Cok f$ are $2\de$-trivial.
Thus, Theorem~\ref{thm:IMT} induces the following theorem (see \citealt[Theorem 4.5]{bauer2015induced}):

\begin{theorem}[AST] \label{thm:AST}
Let $M,N$ be persistence modules in $\rep_{\Bbbk} A_n$.
Then 
\begin{equation}
d_B (\calB(M),\calB(N)) \leq d_I (M,N).
\end{equation}
\end{theorem} 

\begin{proof}
Let $f \colon M \to N(\de)$ be a $\de$-interleaving morphism in $\rep_{\Bbbk} A_n$ and $\ell\geq n$ an integer large enough such that $r_M^\de$ and $r_N^\de$ are bijective.  
Then $M$ and $N$ are regarded as representations of $A_\ell$.
Since $\Ker f$ and $\Cok f$ are $2\de$-trivial, 
\begin{equation}
r_N^{\de} \circ \calB(f) = r_N^{\de} \circ (\calB(f) \circ r_M^{\de}) \circ (r_M^{\de})^{-1} \colon \calB(M) \xrightarrow{\sim} \calB(M(\de)) \nrightarrow \calB(N(\de)) \xrightarrow{\sim} \calB(N)
\end{equation}
is a $\de$-matching by Theorem~\ref{thm:IMT}, as desired.
%\qed
\end{proof}

\section{Derived categories} \label{sec:dercat}

\subsection{Definition and basic properties}  \label{subsec:basic}
Let $\calA$ be an abelian category.
We start this section with the definition of its derived category (see \citealt[Chapter I, 3]{happel1988triangulated}).

\begin{definition} \label{dfn:derived}
(1) A \emph{cochain complex $X^{\bullet}$ over $\calA$} is a family $X^{\bullet}=(X^i,d_X^i )_{i\in \bbZ}$ of objects $X^i$ of $\calA$ and morphisms $d_X^i \colon X^i \to X^{i+1}$ in $\calA$ satisfying $d^{i+1}_X \circ d^i_X=0$.
A cochain complex $X^{\bullet}$ is said to be \emph{bounded} if $X^i=0$ for $|i|\gg 0$.
If $X^i=0$ for $i\not=l$, then it is called a \emph{stalk complex concentrated at the $l$-th term}.
For each cochain complex $X^{\bullet}$ and each $i\in\bbZ$, we have the \emph{$i$-th cohomology} $H^i (X^{\bullet}) := \Ker d_X^i / \im d_X^{i-1}$.

Let $X^{\bullet},Y^{\bullet}$ be cochain complexes over $\calA$.
Then a \emph{cochain map} $f^{\bullet}\colon X^{\bullet}\to Y^{\bullet}$ is a family $f^{\bullet}
=(f^i)_{i\in\bbZ}$ of morphisms $f^i \colon X^i \to Y^i$ in $\calA$ satisfying $f^{i+1} \circ d_X^i =d_Y^i \circ f^i$.

We use $\Cb (\calA)$ to denote the category of bounded cochain complexes and cochain maps over $\calA$.
Then the $l$-translation functor $[l]\colon \Cb(\calA) \to \Cb(\calA)$ is defined by 
$X^{\bullet}[l]:=((X[l])^i, d_{X[l]}^i)_{i\in \bbZ} $ with $(X[l])^i=X^{i+l},$ $d_{X[l]}^i=(-1)^l d_X^{i+l}$ and $(f^{\bullet}[l])^i:= f^{i+l}$ for a cochain complex $X^{\bullet}=(X^i,d_X^i)_{i\in \bbZ}$ and a cochain map $f^{\bullet}$. 
The $0$-translation functor $[0]$ is exactly the identity functor. 

(2) A cochain map $f^{\bullet}$ naturally induces a morphism $H^i (f^{\bullet})\colon H^i (X^{\bullet}) \to H^i (Y^{\bullet})$ for each $i$.
Thus we have the $i$-th cohomology functor $H^i(\blank)\colon \Cb (\calA) \to \calA$.

A cochain map $f^{\bullet}$ is called a \emph{quasi-isomorphism} if the induced morphism $H^i (f^{\bullet})$ is an isomorphism for any $i$.

(3) A cochain map $f^{\bullet} \colon X^{\bullet}\to Y^{\bullet} $  is said to be \emph{null-homotopic} if there exists a family $(h^i)_{i\in\bbZ}$ of morphisms $h^i \colon X^{i} \to Y^{i-1}$ such that  $f^i = h^{i+1} \circ d_X^i + d_Y^{i-1} \circ h^i $ for each $i$. 

Let $I$ be the ideal of $\Cb (\calA )$ consisting of null-homotopic cochain maps.
Then the \emph{bounded homotopy category} $\Kb (\calA)$ of $\calA$ is defined as the quotient category of $\Cb(\calA)$ by the ideal $I$. 
Since a cochain map $f^{\bullet}$ is null-homotopic if and only if $f^{\bullet}[l]$ is so,
we can extend the $l$-translation functor $[l]$ to the setting of the bounded homotopy category such that the following diagram commutes, where $\pi \colon \Cb (\calA) \to \Kb (\calA)$ is the canonical quotient functor:
\begin{equation}
\xymatrix{
\Cb (\calA) \ar[r]^{\pi} \ar[d]_{[l]} & \Kb(\calA) \ar[d]^{[l]} \\
\Cb (\calA) \ar[r]^{\pi} &  \Kb(\calA)
}.
\end{equation}

It is well-known that the homotopy category $\Kb (\calA)$ with the $1$-translation functor $[1]$ forms a triangulated category (see \citealt{happel1988triangulated}).

Moreover, if a cochain map $f^{\bullet}$ is null-homotopic, then $H^i (f^{\bullet})=0$.
Thus, we obtain the $i$-th cohomology functor $H^i(\blank)\colon \Kb (\calA) \to \calA$ such that the following diagram commutes:
\begin{equation}
\xymatrix{
\Cb (\calA) \ar[r]^{\pi} \ar[d]_{H^i} & \Kb(\calA) \ar[dl]^{H^i} \\
\calA
}.
\end{equation}
A \emph{quasi-isomorphism} in $\Kb (\calA)$ is a morphism $f^{\bullet}$ such that $H^i(f^{\bullet})$ is an isomorphism for any $i$.

(4) The \emph{bounded derived category} $\Db (\calA)$ of $\calA$ is the triangulated category given by the Verdier localization of the bounded  homotopy category $\Kb (\calA)$ with respect to the collection of all quasi-isomorphisms. 
Thus, a quasi-isomorphism in $\Kb (\calA)$ defines an isomorphism in $\Db(\calA)$.  
The construction is analogous to that of the localization of a ring.
Indeed, a morphism from $X^{\bullet}$ to $Y^{\bullet}$ in the derived category is represented by a pair $(f^{\bullet},s^{\bullet})$ of a morphism $f^{\bullet}\colon X^{\bullet} \to Z^{\bullet}$ and a quasi-isomorphism $s^{\bullet}\colon Y^{\bullet} \to Z^{\bullet}$ with some cochain complex $Z^{\bullet}$.

By the universal property of the localization, 
we obtain the $i$-th cohomology functor $H^i(\blank)\colon \Db (\calA) \to \calA$ such that the following diagram commutes, where $\iota \colon \Kb (\calA) \to \Db (\calA)$ is the canonical localization functor
\begin{equation}
\xymatrix{
\Kb (\calA) \ar[r]^{\iota} \ar[d]_{H^i} & \Db(\calA) \ar[dl]^{H^i} \\
\calA
}.
\end{equation}
\end{definition}

The abelian category $\calA$ can be regarded as a full subcategory of its bounded derived category $\Db(\calA)$.
Indeed, we have the fully faithful functor $\xi$ from $\calA$ to $\Db(\calA)$ sending $X$ to the stalk complex 
\begin{equation}
\cdots \to 0 \to X \to 0 \to \cdots 
\end{equation}
concentrated at the $0$-th term.
We denote this stalk complex as $X[0]$, using the $0$-translation functor $[0]$.
The stalk complex concentrated at the $l$-th term is written as $X[-l]$.

We use $\proj \calA$ to denote the full subcategory of $\calA$ consisting of projective objects.
An abelian category $\calA$ is said to \emph{have enough projectives} 
if for each object $M \in \calA$, there exists an epimorphism $P \to M$ with $P\in \proj \calA$.
In this case, a \emph{projective resolution} of $M$ is a cochain complex  
\begin{equation}
P^{\bullet} : \cdots \to P_1 \to P_0 
\end{equation}
over $\proj \calA$ satisfying 
\begin{equation}
H^i(P^{\bullet})\cong 
\left\{
\begin{array}{cc}
M,    & i=0 \\
0,    & \text{otherwise}
\end{array}.
\right.
\end{equation}
In other words, we have an exact sequence 
\begin{equation}
\cdots \to P_1 \to P_0 \to M \to 0.
\end{equation}
The \emph{projective dimension} of $M$ is defined as the infimum of the lengths of projective resolutions of $M$, and 
the \emph{global dimension} of $\calA$ is defined as the supremum of all projective dimensions.

Similarly, we use $\inj \calA$ to denote the full subcategory of $\calA$ consisting of injective objects, and we can dually consider the concept of having \emph{enough injectives} and the \emph{injective} and \emph{global dimension}.
Note that the global dimensions defined by projective and injective dimensions coincide if the abelian category $\calA$ has enough projectives and injectives.

The following lemma is important for understanding the bounded derived category (see \citealt[Chapter I, 3.3]{happel1988triangulated}). 
Note that the categories $\Cb (\calA)$ and $\Kb (\calA)$ can be defined if $\calA$ is an additive category, e.g., $\Kb(\proj \calA)$ and $\Kb (\inj \calA)$.

\begin{lemma} \label{lem:fingl}
Assume that $\calA$ has enough projectives (resp.\ injectives) and has a finite global dimension.
Then $\Kb(\proj \calA) (\text{resp.\ $\Kb (\inj \calA)$}) \hookrightarrow \Kb(\calA) \xrightarrow{\iota} \Db(\calA)$ is an equivalence of triangulated categories. 
\end{lemma}

This lemma states that an object of a derived category $\Db (\calA)$ may be assumed to be a cochain complex over $\proj \calA$ (or $\inj \calA$), and a morphism in $\Db (\calA)$ can be written as a morphism in $\Kb (\proj \calA)$ (or $\Kb(\inj \calA)$).
All representatives in a morphism in $\Db (\calA)$ from $X^{\bullet}$ to $Y^{\bullet}$ need some unknown cochain complex $Z^{\bullet}$ (see Definition~\ref{dfn:derived} (4)).  
In contrast, 
a morphism in $\Kb (\proj \calA)$ (or $\Kb(\inj \calA)$) is concretely written as the residue class of a cochain map. 
Thus, under the equivalence in Lemma~\ref{lem:fingl}, we can well understand the bounded derived category $\Db(\calA)$.

Let $\calB$ be another abelian category.
The category $\calA$ is said to be \emph{derived equivalent to $\calB$} if $\Db(\calA)$ and $\Db(\calB)$ are equivalent as triangulated categories.

Let $F \colon \calA \to \calB$ be a functor.
Such a functor yields the canonical functor $C(F)\colon \Cb (\calA) \to \Cb (\calB)$ given by $C(F)(X^{\bullet}):=(FX^i, Fd_X^i)_{i \in \bbZ}$ for each cochain complex $X^{\bullet}=(X^i,d_X^i)_{i \in \bbZ}$.
Then $C(F)$ naturally extends to the functor \linebreak $K(F)\colon \Kb (\calA) \to \Kb(\calB)$.
Moreover, we assume that $F$ is exact.
In this case, $H^i \circ K(F) \cong F \circ H^i$ canonically.
Thus, 
$K(F)$ induces a functor $D(F) \colon \Db (\calA) \to \Db (\calB)$ given by $D(F)(X^{\bullet}):=(FX^i, Fd_X^i)_{i \in \bbZ} $ for each cochain complex $X^{\bullet}=(X^i,d_X^i)_{i \in \bbZ}$.
We often write $C(F)$, $K(F)$, $D(F)$ as $F$ in their respective contents and identify $H^i\circ F$ with $F \circ H^i$.
For example, since the $\de$-shift functor $(\de)\colon \rep_{\Bbbk} A_n \to \rep_{\Bbbk} A_n$ is exact, it induces a functor 
\begin{equation}
    (\de) \colon \Db (\rep_{\Bbbk} A_n) \to \Db (\rep_{\Bbbk} A_n)
\end{equation} 
by $X^{\bullet} (\de) = (X^{i}(\de), d_{X}^{i} (\de))_{i \in \bbZ}$. 
Then $H^i \circ (\de)$ is identified with $(\de) \circ H^i$.

If $\calA$ has enough injectives and $F$ is left exact, then 
we can define the right derived functor $\RF \colon \Db (\calA) \to \Db (\calB)$ by $\RF (X^{\bullet}) := F(I^{\bullet}) $ with $X^{\bullet} \cong I^{\bullet}\in \Kb(\inj \calA)$.
A typical example of a left exact functor is the Hom functor $\Hom (X,\blank)$ with $X\in \calA$. 
Its right derived functor is denoted by $\RHom(X,\blank)$.
Note that $H^i(\RHom(X,\blank))\cong \Ext^i(X,\blank)$ for each $i$.

\subsection{Derived category of $\rep_{\Bbbk} A_n(a)$} \label{subsec:dercateAn}
We discuss some specific properties of $\Db(\rep_{\Bbbk} A_n(a))$ for an $A_n$-type quiver $A_n(a)$ with arbitrary orientation $a$.
Set $\proj A_n(a) := \proj (\rep_{\Bbbk} A_n(a))$.
Note that $\rep_{\Bbbk} A_n(a)$ has enough projectives and injectives, and has global dimension $0$ for $n=1$ and $1$ for $n>1$. 
Hence we always have a projective resolution $0\to P_1 \to P_0 \to M \to 0$ of length  at most $1$ for any representation $M\in \rep_{\Bbbk} A_n(a)$.
In particular, any subrepresentation of a projective representation is also projective in this setting.
In addition, by Lemma~\ref{lem:fingl}, we obtain an equivalence between $\Db(\rep_{\Bbbk} A_n(a))$ and $\Kb(\proj A_n(a))$.

For $\Db(\rep_{\Bbbk} A_n (a))$, we have the following strong characterization of a cochain complex by its cohomologies.

\begin{lemma} \label{lem:dec}
For any cochain complex $X^{\bullet} \in \Db (\rep_{\Bbbk} A_n (a))$, 
\begin{equation}X^{\bullet}\cong \Ds_{i \in \bbZ} H^i (X^{\bullet}) [-i]\end{equation} in $\Db(\rep_{\Bbbk} A_n (a))$.
More generally, 
\begin{equation}
    X^{\bullet} (\de)\cong 
    \Ds_{i \in \bbZ} H^i (X^{\bullet})(\de) [-i].
\end{equation}
\end{lemma}
\begin{proof}
We may assume that $X^{i}$ is projective for all $i \in \bbZ$, since $ \Db (\rep_{\Bbbk} A_n (a)) \simeq \Kb (\proj A_n (a))$ by Lemma~\ref{lem:fingl}.
The statement will be proved by induction on the length of its non-zero terms.
By the boundedness of $X^{\bullet}$,
let us set $s:=\min\{i\in \bbZ \mid X^i \not= 0 \}$ and $t:= \max\{i\in \bbZ \mid X^i \not= 0 \}$. 
Then the length $\ell (X^{\bullet})$ of non-zero terms of $X^{\bullet}$ is defined by $t-s+1$.
Since the global dimension of $\rep_{\Bbbk} A_n (a)$ is at most 1, $\im d_X^{t-1}$  is projective. Hence the exact sequence 
\begin{equation}
0 \to \Ker d_X^{t-1} \to X^{t-1} \to \im d_X^{t-1} \to 0
\end{equation}
is split, implying that $\Ker d_X^{t-1}$ is also projective.
Thus, we have $X^{\bullet} \cong Y^{\bullet} \ds Z^{\bullet}$, 
where the complexes $ Y^{\bullet}, Z^{\bullet}\in \Kb(\proj A_n (a))$ are given by
\begin{equation}
Y^{\bullet}= \cdots \to X^{t-2} \to \Ker d_X^{t-1} \to 0 \text{ and } Z^{\bullet} = \cdots \to 0 \to \im d_X^{t-1} \to X^t.
\end{equation}
Here, $Z^{\bullet} \cong H^{t}(X^{\bullet})[-t]$ and
\begin{equation}
    H^{i}(Y^{\bullet}) \cong \left. \begin{cases}
    H^{i}(X^{\bullet}), & \text{if } i<t \\
    0, & \text{if } i\geq t
    \end{cases}. \right. 
\end{equation}
The length $\ell (Y^{\bullet})$ is less than the length $\ell (X^{\bullet})$, 
so by induction 
\begin{equation}Y^{\bullet} \cong \Ds_{i < t} H^{i}(Y^{\bullet})[-i] \cong \Ds_{i<t} H^{i}(X^{\bullet})[-i].\end{equation} 
Therefore, we obtain $X^{\bullet} \cong  Y^{\bullet} \ds Z^{\bullet} \cong \Ds_{i\in \bbZ} H^{i}(X^{\bullet})[-i]$ in $\Db(\rep_{\Bbbk} A_n (a))$. 

The second statement follows from the first one and the fact that the functor $(\de)$ is exact.
%\qed
\end{proof}

The fact that the global dimension of $\rep_{\Bbbk} A_n (a)$ is at most $1$ is essential to the proof of the foregoing lemma.
As a consequence of Lemma~\ref{lem:dec}, 
we can characterize all indecomposable objects of $\Db (\rep_{\Bbbk} A_n(a))$.

\begin{corollary} \label{cor:dec}
A cochain complex $X^{\bullet} \in \Db (\rep_{\Bbbk} A_n(a))$ is indecomposable if and only if $X^{\bullet}$ is isomorphic to a stalk complex 
\begin{equation}
\bbI [b,d] [-i] : \cdots \to 0 \to \bbI [b,d] \to 0 \to \cdots
\end{equation}
concentrated at the $i$-th term in $\Db (\rep_{\Bbbk} A_n(a))$ for some $1\leq b \leq d \leq n$ and some $i\in \bbZ$. 
Thus, any cochain complex $X^{\bullet}$ is isomorphic to 
\begin{equation}
\bigoplus\nolimits_{b\leq d,i} (\bbI [b,d] [-i])^{m(b,d,i)} ,
\end{equation}
where the non-negative integer $m(b,d,i)$ is the multiplicity of $\bbI [b,d] [-i]$.
\end{corollary}

Since $\Db (\rep_{\Bbbk} A_n(a))$ is a Krull-Schmidt category (see \citealt{chen2008algebras}), the interval decomposition in the corollary above is unique.
By using this result, we propose the notion of a `derived' barcode. 
Note that \citet{berkouk2018derived} considered a similar concept of a graded barcode.

\begin{definition} \label{dfn:dbarcode}
Let $X^{\bullet},Y^{\bullet}$ be cochain complexes in $\Db(\rep_{\Bbbk} A_n(a))$.
Then the \emph{derived barcode} $\calB^D(X^{\bullet})$ is defined as
\begin{equation}
\calB^D(X^{\bullet}):= \bigsqcup_{i \text{ with } H^i (X^{\bullet}) \not = 0  } \calB(H^i (X^{\bullet}))
\end{equation}
where $\calB(H^i (X^{\bullet}))$ is the ordinary barcode of $H^i (X^{\bullet})$ (see the paragraph following  Theorem~\ref{thm:gab}).
\end{definition}

As in Section \ref{subsec:intbar}, 
the AR quiver of $\Db (\rep_{\Bbbk} A_n(a))$ with arbitrary orientations can be  defined (see \citealt[Chapter I, 4 and 5]{happel1988triangulated}).
Similar to the case of $\rep_{\Bbbk} A_n(a)$, the derived barcode of $X^{\bullet}$ can be defined as a map $\AR_0 \to \bbZ$ sending $\bbI [b,d] [-i]$ to the multiplicity $m(b,d,i)$, where $\AR_0$ is the set of vertices of the AR quiver of $\Db (\rep_{\Bbbk} A_n(a))$.
Thus, AR quivers are hidden behind barcodes also in this setting. 
Moreover, the AR quiver of $\Db (\rep_{\Bbbk} A_n(a))$ consists of all shifted copies of the AR quiver $\AR(A_n(a))$ of $A_n(a)$.

\begin{example}
The AR quiver $\AR (\Db (\rep_{\Bbbk} A_3))$ is
\[
\AR (\Db (\rep_{\Bbbk} A_3))=
    \begin{tikzcd}[graphstyle,every matrix/.append style={name=m}]
      \cdots&\bbI [1,1] [-1] \arrow[dr] && \bbI [1,3] \arrow[dr] &&  \bbI [3,3] [1]\\
      \bbI[1,2] [-1] \arrow[ur] \arrow[dr]&& \bbI [2,3] \arrow[ur] \arrow[dr] & & \bbI [1,2] \arrow[dr] \arrow[ur]& \cdots  \\
      \cdots&\bbI [3,3] \arrow[ur] &   &\bbI [2,2] \arrow[ur] &  & \bbI [1,1]  
    \end{tikzcd}.
\]
More generally, the AR quiver $\AR (\Db (\rep_{\Bbbk} A_n))$ is as described in Figure~\ref{fig:DAR}.
\begin{figure}[h]
    \centering
    \includegraphics[width=0.9\textwidth,height=0.15\textheight]{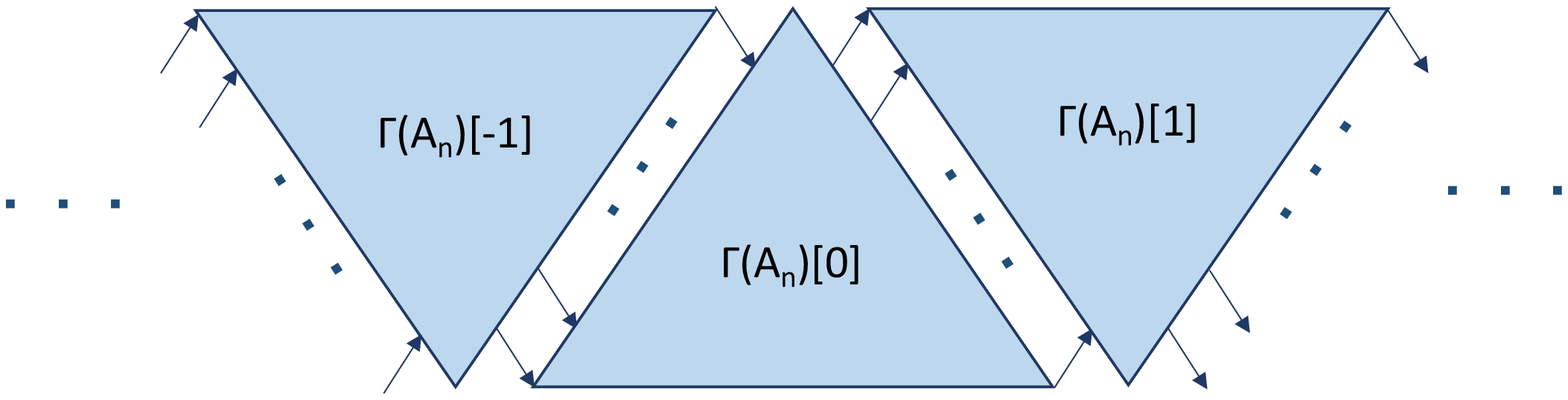}
    \caption{AR quiver of $\Db (\rep_{\Bbbk} A_n)$}
    \label{fig:DAR}
\end{figure}
\end{example}

Moreover, the AR quiver $\AR (\Db (\rep_{\Bbbk} A_3(z_1)))$, where $A_3(z_1):1\leftarrow 2 \to 3$, is
\begin{equation}
\AR (\Db (\rep_{\Bbbk} A_3(z_1)))=
    \begin{tikzcd}[graphstyle,every matrix/.append style={name=m}]
      \cdots&\bbI [1,1] \arrow[dr] && \bbI [2,3] \arrow[dr] &&  \bbI [3,3] [1]\\
      \bbI[2,2][-1] \arrow[ur] \arrow[dr]  && \bbI [1,3] \arrow[ur] \arrow[dr] & & \bbI [2,2] \arrow[dr] \arrow[ur]& \cdots  \\
      \cdots&\bbI [3,3] \arrow[ur] &   &\bbI [1,2] \arrow[ur] &  & \bbI [1,1] [1]  
    \end{tikzcd}.
\end{equation}
Similarly to Figure~\ref{fig:DAR}, the AR quiver $\AR (\Db (\rep_{\Bbbk} A_n(z_1)))$ is given by all shifted copies of the AR quiver of $A_n(z_1)$.

\subsection{Derived interleaving and bottleneck distances}
\label{subsec:deriveddist}

In this subsection, we propose distances on the derived category of persistence modules by extending the original interleaving and bottleneck distances.

Recall that the $\de$-shift functor $(\de)\colon \rep_{\Bbbk} A_n \to \rep_{\Bbbk} A_n$ induces a functor 
\begin{equation}
(\de) \colon \Db (\rep_{\Bbbk} A_n) \to \Db (\rep_{\Bbbk} A_n)
\end{equation} 
via $X^{\bullet} (\de) = (X^{i}(\de), d_{X}^{i} (\de))_{i\in\bbZ}$ 
since the functor $(\de)$ is exact.

\begin{definition} \label{dfn:dinterleave}
Let $X^{\bullet}, Y^{\bullet}$ be cochain complexes in $\Db (\rep_{\Bbbk} A_n)$ and $\de$ a non-negative integer.
Then $X^{\bullet}$ and $Y^{\bullet}$ are said to be  \emph{derived $\de$-interleaved}
if there exist morphisms $f^{\bullet}\colon X^{\bullet} \to Y^{\bullet}(\de)$ and $g^{\bullet} \colon Y^{\bullet} \to X^{\bullet} (\de)$ such that 
for each $i\in \bbZ$, 
$(H^i(f^{\bullet}), H^i(g^{\bullet}))$ is a $\de$-interleaving pair between $H^i(X^{\bullet})$ and $H^i(Y^{\bullet})$ in the sense of Definition~\ref{dfn:interleave}. 
Namely, the following diagrams commute for each $i \in \bbZ$: 
\begin{equation}
\xymatrix@C=8pt{
H^{i} (X^{\bullet}) \ar[rr]^{\ph_{H^{i}(X^{\bullet})}^{2\de} } \ar[dr]_{H^{i}(f^{\bullet})} && H^{i} (X^{\bullet}) (2\de), & H^{i} (Y^{\bullet}) \ar[rr]^{\ph_{H^{i}(Y^{\bullet})}^{2\de}} \ar[dr]_{H^{i}(g^{\bullet})} && H^{i} (Y^{\bullet})(2\de)\\
& H^{i} (Y^{\bullet}) (\de) \ar[ur]_{H^{i}(g^{\bullet})(\de)} & & & H^{i} (X^{\bullet}) (\de) \ar[ur]_{H^{i}(f^{\bullet})(\de)}&
}.
\end{equation}
In this case we also call the pair of $f^{\bullet}\colon X^{\bullet} \to Y^{\bullet}(\de)$ and $g^{\bullet}\colon Y^{\bullet} \to X^{\bullet}(\de)$ a \emph{derived $\de$-interleaving pair}. 
Moreover, we call a morphism $f^{\bullet} \colon X^{\bullet} \to Y^{\bullet}(\de)$ a \emph{derived $\de$-interleaving morphism} if there is a morphism $g^{\bullet}\colon Y^{\bullet} \to X^{\bullet}(\de)$ such that the pair $(f^{\bullet},g^{\bullet})$ is a derived $\de$-interleaving pair.

For cochain complexes $X^{\bullet},Y^{\bullet}$ in $\Db (\rep_{\Bbbk} A_n)$, the \emph{derived interleaving distance} is defined as
\begin{equation}d_I^D (X^{\bullet},Y^{\bullet}) := \inf\{ \de \in \bbZ_{\geq 0} \mid X^{\bullet} \text{ and } Y^{\bullet} \text{ are derived $\de$-interleaved} \}.\end{equation}
\end{definition}

\begin{remark} \label{rmk:dint}
Similarly to the original setting, $d_I^D(X^{\bullet},Y^{\bullet})=0$ for two cochain complexes $X^{\bullet}, Y^{\bullet}\in \Db (\rep_{\Bbbk} A_n)$ if and only if $X^{\bullet}$ and $Y^{\bullet}$ are isomorphic in $\Db(\rep_{\Bbbk} A_n)$.
Thus, the derived interleaving distance also measures how far these complexes are from being isomorphic.
\end{remark}

Note that if $X^{\bullet}$ and $Y^{\bullet}$ are derived $\de$-interleaved,
then $H^i(X^{\bullet})$ and $H^i(Y^{\bullet})$ are $\de$-interleaved for all $i$.
It follows from Lemma~\ref{lem:dec} that the converse also holds.

\begin{corollary} \label{cor:dint}
Let $X^{\bullet}$, $Y^{\bullet}$ be cochain complexes in $\Db (\rep_{\Bbbk} A_n)$.
Then $X^{\bullet}$ and $Y^{\bullet}$ are derived $\de$-interleaved if and only if
$H^i(X^{\bullet})$ and $H^i(Y^{\bullet})$ are $\de$-interleaved for all $i$.
\end{corollary}
\begin{proof}
For each $i$, let $(f_i,g_i)$ be a $\de$-interleaved pair between $H^i(X^{\bullet})$ and $H^i(Y^{\bullet})$.
Then the pair of 
\begin{equation}
    (f_i[-i]) \colon \bigoplus\nolimits_{i} H^i(X^{\bullet})[-i] \to \bigoplus\nolimits_{i} H^i(Y^{\bullet})(\de)[-i]
\end{equation} 
and 
\begin{equation}
    (g_i[-i]) \colon \bigoplus\nolimits_{i} H^i(Y^{\bullet})[-i] \to \bigoplus\nolimits_{i} H^i(X^{\bullet})(\de)[-i] 
\end{equation}
is a derived $\de$-interleaved pair.
By Lemma~\ref{lem:dec}, $X^{\bullet}$ and $Y^{\bullet}$ are derived $\de$-interleaved. 
Indeed, using the isomorphisms in Lemma~\ref{lem:dec}, 
we can construct a derived $\de$-interleaving pair of morphisms 
\begin{equation}
    f^{\bullet} \colon X^{\bullet} \to \bigoplus\nolimits_{i} H^i(X^{\bullet})[-i] \xrightarrow{(f_i[-i])} \bigoplus\nolimits_{i} H^i(Y^{\bullet})(\de)[-i] \to Y^{\bullet} (\de)
\end{equation}
and
\begin{equation}
    g^{\bullet} \colon Y^{\bullet} \to \bigoplus\nolimits_{i} H^i(Y^{\bullet})[-i] \xrightarrow{(g_i[-i])} \bigoplus\nolimits_{i} H^i(X^{\bullet})(\de)[-i]  \to X^{\bullet} (\de)
\end{equation}
such that $H^i(f^{\bullet}) = f_i$ and $H^i(g^{\bullet})=g_i$ for any $i$.
%\qed
\end{proof}

Finally, we propose the `derived' bottleneck distance between derived barcodes in the sense of Definition~\ref{dfn:dbarcode} in this setting.

\begin{definition} \label{dfn:dbottleneck}
Let $X^{\bullet},Y^{\bullet}$ be cochain complexes in $\Db(\rep_{\Bbbk} A_n)$.
Two derived barcodes $\calB^D (X^{\bullet})$ and $\calB^D (Y^{\bullet})$ are said to be \emph{$\de$-matched}
if $B (H^i(X^{\bullet}))$ and $B (H^i(Y^{\bullet}))$ are $\de$-matched in the sense of Definition~\ref{dfn:bottleneck} for all $i\in\bbZ$.

For derived barcodes $\calB^D(X^{\bullet}),\calB^D(Y^{\bullet})$, the \emph{derived bottleneck distance} is defined as 
\begin{equation}
d_B^D (\calB^D(X^{\bullet}),\calB^D(Y^{\bullet})) := 
\inf \{ \de \in \bbZ_{\geq 0} \mid \text{$\calB^D (X^{\bullet})$ and $\calB^D (Y^{\bullet})$ are  $\de$-matched} \}. 
\end{equation}
\end{definition}

\section{Main results} \label{sec:main}

In this section, we derive an AST for zigzag persistence modules from an AST for ordinary ones by using the derived category.
We adopt a different approach from that of \citet{botnan2018algebraic}.
In fact, we consider the distances on zigzag persistence modules naturally induced by the known interleaving and bottleneck distances on ordinary ones using derived categories.
This enables us to obtain an AST for a wider class compared to that of Botnan and Lesnick.

\subsection{AST for derived categories} \label{subsec:DAST}

In this subsection, we first prove an AST for derived categories of ordinary persistence modules.

\begin{theorem}[AST for derived categories] \label{thm:DAST}
Let $X^{\bullet},Y^{\bullet}$ be cochain complexes in $\Db (\rep_{\Bbbk} A_n)$.
Then
\begin{equation}d_B^D (\calB^D (X^{\bullet}),\calB^D (Y^{\bullet})) \leq d_I^D (X^{\bullet},Y^{\bullet}).\end{equation}  
\end{theorem}

\begin{proof}%[Proof of Theorem~\ref{thm:DAST}]
Assume that $X^{\bullet}$ and $Y^{\bullet}$ are derived $\de$-interleaved.
Then for all $i\in \bbZ$, $H^i(X^{\bullet})$ and  $H^i(Y^{\bullet})$ are $\de$-interleaved, and hence
$\calB(H^i (X^{\bullet}))$ and $\calB(H^i (Y^{\bullet}))$ are $\de$-matched by Theorem~\ref{thm:AST}.
Thus, by definition, the inequality
\begin{equation}d_B^D (\calB^D (X^{\bullet}),\calB^D (Y^{\bullet})) \leq d_I^D (X^{\bullet},Y^{\bullet})\end{equation}
holds.
%\qed
\end{proof}

Next, we consider a distance induced by a derived equivalence.
Let $\calA$ be an abelian category. 
Assume that there exists a derived equivalence  $E$ from $\Db(\calA)$ to $\Db(\rep_{\Bbbk} A_n)$.
From \citet{rickard1989morita}, the derived equivalence can be characterized by a tilting complex. 
The construction of the tilting complex is not so easy, but a special one will be utilized in this paper (see Subsection~\ref{subsec:ZAST}).

\begin{definition} \label{dfn:dereqint}
Two objects $X$ and $Y$ of $\Db(\calA)$ are said to be \emph{$\de$-interleaved with respect to $E$} if 
$E(X)$ and $E(Y)$ are derived $\de$-interleaved in the sense of Definition~\ref{dfn:dinterleave}.
The \emph{interleaving distance $d_I^{E,\calA} (X,Y)$ with respect to $E$} is defined 
as
\begin{equation}
d_I^{E,\calA} (X,Y):=\inf\{ \de \in \bbZ_{\geq 0} \mid \text{$X$ and $Y$ are $\de$-interleaved with respect to $E$}\}.
\end{equation}
Namely, $d_I^{E,\calA} (X,Y)=d_I^D (E(X),E(Y))$ holds.
\end{definition}

\begin{remark}
By Remark~\ref{rmk:dint}, $d_I^{E,\calA} (X,Y)=0$ if and only if $E(X)$ and $E(Y)$ are isomorphic in $\Db (\rep_{\Bbbk} A_n)$.
Since $E$ is an equivalence, this means that $X$ and $Y$ are isomorphic in $\Db (\calA)$.
Thus, the interleaving distance defined as above also measures how far these objects are from being isomorphic.
This justifies calling the distance an interleaving distance. 
\end{remark}
\begin{remark}
The $\delta$-shift functor cannot be defined in the zigzag setting, so neither can the usual interleaving distance. 
One of the advantages of our approach is that we can define the interleaving distance even in the zigzag setting through the derived equivalence.
\end{remark}

Since $E$ is an equivalence, in particular, a fully faithful functor, 
$X\in \Db(\calA)$ is indecomposable if and only if so is $E(X)\in \Db(\rep_{\Bbbk} A_n)$.
Hence, since $\Db(\rep_{\Bbbk} A_n)$ is a Krull-Schmidt category, so is $\Db(\calA)$. 
Consequently, the derived equivalence $E$ induces a bijection between $\calB^D(E(X))$ (see Definition~\ref{dfn:dbarcode}) and  
\begin{equation}
    \calB^D_{\calA} (X):=\{ Z\in \Db(\calA) \mid \text{$Z$ is indecomposable and a direct summand of $X$} \}.
\end{equation}
Then the following distance between $\calB^D_{\calA} (X)$ and $\calB^D_{\calA} (Y)$ is naturally derived by passing through the derived equivalence $E$.

\begin{definition} \label{dfn:dereqbottle}
For two objects $X,Y$ of $\Db(\calA)$,
$\calB^D_{\calA}(X)$ and $\calB^D_{\calA}(Y)$ are said to be  \emph{$\de$-matched with respect to $E$} if $\calB^D(E(X))$ and $\calB^D(E(Y))$ are $\de$-matched in the sense of Definition~\ref{dfn:dbottleneck}.
The \emph{bottleneck distance $d_B^{E,\calA} (\calB^D_{\calA}(X),\calB^D_{\calA}(Y))$ with respect to $E$} is defined as
\begin{equation}
d_B^{E,\calA} (\calB^D_{\calA}(X),\calB^D_{\calA}(Y)):=
\inf \left\{ \de \in \bbZ_{\geq 0}\ \middle|\ 
\begin{array}{c}
\text{$\calB^D_{\calA} (X)$ and $\calB^D_{\calA} (Y)$ are }\\
\text{$\de$-matched with respect to $E$}
\end{array}
\right\}. 
\end{equation}
Namely, $d_B^{E,\calA} (\calB^D_{\calA}(X),\calB^D_{\calA}(Y))=d_B^D (\calB^D(E(X)),\calB^D(E(Y)))$ holds.
\end{definition}

In our convention, \emph{an AST} states that the interleaving distance between objects $X$ and $Y$ gives an upper bound for the bottleneck distance between their barcodes. 
Thus, as a consequence of Theorem~\ref{thm:DAST}, Definition~\ref{dfn:dereqint}, and 
Definition~\ref{dfn:dereqbottle}, we have the following AST for the derived category $\Db(\calA)$.

\begin{proposition} \label{prp:ASTdereq}
Let $\calA$ be an abelian category and $X,Y$ objects in $\Db(\calA)$.
Assume that there exists a derived equivalence $E$ from $\Db(\calA)$ to $\Db(\rep_{\Bbbk} A_n)$. 
Then 
\begin{equation}
d_B^{E,\calA} (\calB_{\calA}^D(X),\calB_{\calA}^D(Y)) \leq d_I^{E,\calA} (X,Y).
\end{equation}
\end{proposition}

We will construct an equivalence $E$ explicitly in the next section, using classical tilting modules.

\subsection{AST for zigzag persistence modules} \label{subsec:ZAST}

In this subsection, we first provide an AST for an abelian category $\calA$ which is derived equivalent to $\rep_{\Bbbk} A_n$.
Recall that $\calA$ can be regarded as a full subcategory of $\Db(\calA)$.
As a consequence of Proposition~\ref{prp:ASTdereq}, we have the following result.

\begin{corollary} \label{cor:ASTdereq}
Let $\calA$ be an abelian category. 
Assume that $\calA$ is derived equivalent to $\rep_{\Bbbk} A_n$.
Then an AST also holds for $\calA$.
\end{corollary}

Next, we discuss an AST for zigzag persistence modules through a derived equivalence  between the categories $\Db(\rep_{\Bbbk} A_n(a))$ and $\Db(\rep_{\Bbbk} A_n )$ for an $A_n$-type quiver $A_n(a)$ with arbitrary orientation $a$. 
We consider the derived equivalence given by a classical tilting module.

\begin{definition}[\citealt{assem2006elements,brenner1980generalizations,bongartz1981tilted,happel1982tilted}]
Let $T$ be a persistence module.
Then $T$ is called a \emph{classical tilting module} if it satisfies the following three conditions:
\begin{enumerate}
    \item[(1)] the projective dimension of $T$ is at most $1$,
    \item[(2)] $\Ext^i (T,T)=0$ for all $i>0$, and 
    \item[(3)] $T$ has exactly $n$ non-isomorphic indecomposable direct summands. 
\end{enumerate}
\end{definition}

\citet{bongartz1981tilted} proved that our classical tilting modules are equivalent to the original ones 
(see \citealt{bongartz1981tilted} or \citealt{happel1982tilted} for the original definition), which are exactly tilting modules with projective dimension at most $1$ in the sense of \citet{miyashita1986tilting} (see also \citealt[p.118]{happel1988triangulated}).
Here, we recall that $\Hom(T,\blank)$ is a functor from $\rep_{\Bbbk} A_n$ to the module category $\mod \End (T)^{\op}$ of the endomorphism algebra on $T$.
We have the following property of classical tilting modules. 

\begin{lemma}[{\citealt[Chapter III]{happel1988triangulated}}] \label{lem:tiltderived}
Let $T$ be a classical tilting module.
Assume that the endomorphism algebra $B=\End (T)^{\op}$ is presented by the quiver $Q_B$ with no relations.
Then the functor $\RHom(T,\blank)$ is a derived equivalence  from $\Db(\rep_{\Bbbk} A_n)$ to $\Db(\rep_{\Bbbk} Q_B)$.
\end{lemma}

Now, we construct a classical tilting module whose endomorphism algebra is presented by the quiver $A_n(a)$ (with no relations).

Let $\tau, \tau^{-1}$ be the AR translations in $\rep_{\Bbbk} A_n$ (see 
\citealt[Chapter IV.2, 2.3 Definition]{assem2006elements} or \citealt[2.3.3]{schiffler2014quiver} for definition).
For an indecomposable non-projective (resp.\ non-injective) representation $M$,  
$\tau (M)$ is indecomposable non-injective and $\tau^{-1}\tau(M)\cong M$ (resp.\ $\tau^{-1} (M)$ is indecomposable non-projective and $\tau\tau^{-1}(M)\cong M$) (see \citealt[Chapter IV.2, 2.10 Proposition]{assem2006elements}). 
The \emph{$\tau$-orbit} of $M$ is the set of indecomposable representations of the form $\tau^m(M)$ or $\tau^{-m}(M)$ for some non-negative integer $m$, where $\tau^{-m}:=(\tau^{-1})^m$. 

Since the AR quiver of $A_n$ is finite and connected, 
there are finitely many $\tau$-orbits.
Moreover, each $\tau$-orbit contains exactly one indecomposable projective representation (see \citealt[3.1.2]{schiffler2014quiver}).

Let $O(P_i)$ be the $\tau$-orbit of the indecomposable projective representation $P_i$ corresponding to the vertex $i$ of $A_n$.
Note that there are $n$ $\tau$-orbits $O(P_1), \cdots, O(P_n)$, and that $O(P_1)=\{P_1\}$ since $P_1$ is projective-injective.
Then, the set $\{O(P_i)\mid i=1,\cdots, n\}$ is just the set of all $\tau$-orbits,
the $\tau$-orbit $O(P_i)$ is finite for any $i$, and 
any indecomposable representation of $A_n$ belongs to the $\tau$-orbit $O(P_i)$ for some $i$. 
A \emph{section} of the AR quiver $\AR(A_n)$ is a connected full subquiver formed by representatives of all $\tau$-orbits $O(P_i)$.

\begin{example}
In the AR quiver $\AR (A_3)$ 
\begin{equation}
\AR (A_3)=
    \begin{tikzcd}[graphstyle,every matrix/.append style={name=m}]%,cells={nodes={draw,minimum height=0.6cm}}]
      && \boxed{\bbI [1,3]} \arrow[dr] \\
      & \boxed{\bbI [2,3]} \arrow[ur] \arrow[dr] & 
      & \arrow[ll,dashed] \bbI [1,2] \arrow[dr]  \\
      \bbI [3,3] \arrow[ur]  &   & \arrow[ll,dashed] \boxed{\bbI [2,2]} \arrow[ur] &  & \arrow[ll,dashed]\bbI [1,1] 
    \end{tikzcd},
\end{equation}
the actions of $\tau$ are denoted by dashed arrows, meaning that there are three $\tau$-orbits 
$O(\bbI[3,3])=\{\bbI[3,3],\bbI[2,2],\bbI[1,1]\}$,
$O(\bbI[2,3])=\{\bbI[2,3], \bbI[1,2]\}$, and 
$O(\bbI[1,3])=\{\bbI[1,3]\}$ 
of projective representations $\bbI[3,3],\bbI[2,3]$, and $\bbI[1,3]$, respectively.
In this case, for example, $\Sigma=\{\bbI[1,3],\bbI[2,3],\bbI[2,2] \}$ is a section.
These representations are written as rectangle-surrounded vertices in the foregoing AR quiver.

More generally, a section $\Sigma$ in the AR quiver $\AR ( A_n )$ of $A_n$ is described like the red line in Figure~\ref{fig:section}.
\begin{figure}[h]
    \centering
    \includegraphics[width=4cm,height=4cm]{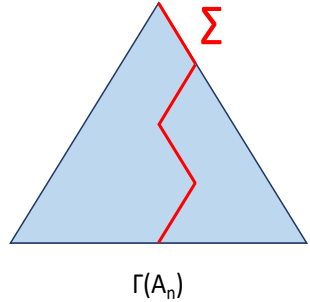}%{section.png}
    \caption{Example of a section $\Si$ in the AR quiver $\AR(A_n)$}
    \label{fig:section}
\end{figure}
\end{example}

Fix a section $\Sigma$ with vertices
$\Sigma_0=\{X_1=P_1,\cdots,X_n\}$, where $X_i\in O(P_i)$ for each $i$. 
Note that $P_1$ is the top vertex of the AR quiver $\AR (A_n)$ of $A_n$. 
Then we define a module $T(\Si)$ as follows:
\begin{equation}
T(\Si) = \bigoplus_{i=1}^{n} X_i.
\end{equation}

\begin{lemma}[{\citealt[2.6 Corollary]{bongartz1981tilted}}]
\label{lem:tiltsection}
For any section $\Si$, $T(\Si)$ is a classical tilting module in $\rep_{\Bbbk} A_n$.
\end{lemma}

In this setting, $P_1$ is a direct summand of any classical tilting module.
More generally, any projective-injective representation is a direct summand of any classical tilting module.

The endomorphism algebra $\End (T(\Si))^{\op}$ is presented by the quiver $\Sigma^{\op}$.
By definition, every section is an $A_n(a)$-type quiver with some orientation $a$ and any $A_n(a)$-type quiver appears as a section. 
As a consequence of Lemma~\ref{lem:tiltderived} and Lemma~\ref{lem:tiltsection} we have the following result.

\begin{proposition} \label{prp:tiltderived}
If a section $\Si$ is isomorphic to $A_n(a)^{\op}$,
then the functor $\RHom(T(\Si),\blank)$ is a derived equivalence from $\Db(\rep_{\Bbbk} A_n)$ to $\Db(\rep_{\Bbbk} A_n(a))$.
\end{proposition}

\begin{remark}
Under the derived equivalence $\RHom(T(\Si),\blank)$,
the AR quiver $\AR(A_n(a))$ can be embedded into the AR quiver of $\Db(\rep_{\Bbbk} A_n)$ as in Figure~\ref{fig:zigzag}.
\begin{figure}[h]
    \centering
    \includegraphics[width=7.5cm,height=5cm]{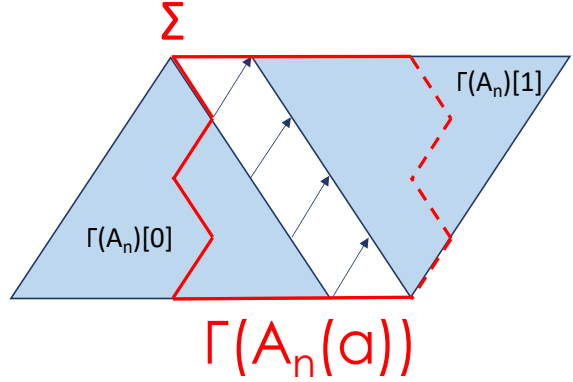}
    \caption{The red-bordered polygon except for the broken polygonal line is the area of the AR quiver $\AR (A_n(a))$.
    For example, the AR quiver $\AR (A_3(z_1))$ was described as in Example~\ref{ex:ARquiver} 
    }
    \label{fig:zigzag}
\end{figure}
\end{remark}

In the case that $\calA=\rep_{\Bbbk} A_n(a)$ and $E$ is the quasi-inverse of $\RHom(T(\Si),\blank)$ with $\Si^{\op} \cong A_n (a)$, 
we put $d_I^{T,a} := d_I^{E,\calA}$ and $d_B^{T,a}:=d_B^{E,\calA}$.
Thus, we conclude the following AST for $\rep_{\Bbbk} A_n (a)$ with arbitrary orientations $a$.
Note that $\calB_\calA^D(X)$ in Definition~\ref{dfn:dereqbottle} and the ordinary barcode $\calB(X)$ coincide for any $X \in \calA=\rep_{\Bbbk} A_n(a)$.

\begin{theorem}[AST for zigzag] \label{thm:ASTzigzag}
Let $X,Y$ be zigzag persistence modules in the category  $\rep_{\Bbbk} A_n (a)$ with arbitrary orientation $a$.
Then 
\begin{equation}
d_B^{T,a} (\calB(X), \calB(Y)) \leq d_I^{T,a} (X,Y).
\end{equation}
\end{theorem}

\begin{proof}
Since we have the derived equivalence $\RHom(T(\Si),\blank)$ from $\Db(\rep_{\Bbbk} A_n)$ to $\Db(\rep_{\Bbbk} A_n(a))$ by Proposition~\ref{prp:tiltderived}, 
the result follows from Corollary~\ref{cor:ASTdereq}.
%\qed
\end{proof}

\section{Isometry theorem} \label{sec:DIST}

In this section, we will prove an isometry theorem for the category  $\rep_{\Bbbk} A_n (a)$ of zigzag persistence modules.
Theorem~\ref{thm:AST} gives 
the inequality $d_B \leq d_I$,
which is a part of the following isometry theorem (see \citealt[Theorem 3.1 and Section B.1]{bauer2015induced}).

\begin{theorem}[Isometry theorem] \label{thm:IST}
Let $M,N$ be persistence modules.
Then 
\begin{equation}
d_B (\calB(M),\calB(N)) = d_I (M,N).
\end{equation}
\end{theorem}

By Theorem \ref{thm:IST}, we obtain an isometry theorem for the derived category of persistence modules.

\begin{theorem}[Isometry theorem for derived categories] \label{thm:DIST}
Let $X^{\bullet},Y^{\bullet}$ be cochain complexes in $\Db (\rep_{\Bbbk} A_n)$.
Then 
\begin{equation}
d_B^D (\calB^D(X^{\bullet}),\calB^D(Y^{\bullet})) = d_I^D (X^{\bullet},Y^{\bullet}).
\end{equation}
\end{theorem}
\begin{proof}
From Theorem~\ref{thm:DAST}, we have only to show
\begin{equation}
d_B^D (\calB^D(X^{\bullet}),\calB^D(Y^{\bullet})) \geq d_I^D (X^{\bullet},Y^{\bullet}).
\end{equation}

If $\calB^D(X^{\bullet})$ and $\calB^D(Y^{\bullet})$ are $\de$-matched, 
then there exists a $\de$-matching between $\calB(H^i (X^{\bullet}))$ and $\calB(H^i(Y^{\bullet}))$ for each $i$ by definition.
Then, by Theorem~\ref{thm:IST}, 
there exists a $\de$-interleaving pair $(f_i,g_i)$ between $H^i(X^{\bullet})$ and $H^i (Y^{\bullet})$ for each $i$.
By Corollary~\ref{cor:dint}, $X^{\bullet}$ and $Y^{\bullet}$ are derived $\de$-interleaved in $\Db(\rep_{\Bbbk} A_n)$.
%\qed
\end{proof}

As a consequence of Theorem~\ref{thm:DIST}, we can extend Proposition~\ref{prp:ASTdereq} and Corollary~\ref{cor:ASTdereq} to isometry theorems by Definition~\ref{dfn:dereqint} and Definition~\ref{dfn:dereqbottle}.

\begin{corollary} \label{cor:ISTdersub}
Let $\calA$ be an abelian category and 
$X,Y$ objects in $\calA$ or $\Db(\calA)$.
Assume that there exists a derived equivalence $E$ from $\Db(\calA)$ to $\Db(\rep_{\Bbbk} A_n)$. 
Then 
\begin{equation}
d_B^{E,\calA} (B_{\calA}^D(X),B_{\calA}^D(Y)) = d_I^{E,\calA} (X,Y).
\end{equation}
\end{corollary}

Finally, we can extend Theorem~\ref{thm:ASTzigzag} by Corollary~\ref{cor:ISTdersub} as follows.
\begin{theorem}[Isometry theorem for zigzag] \label{thm:ISTzigzag}
Let $X,Y$ be zigzag persistence modules in the category  $\rep_{\Bbbk} A_n (a)$ with arbitrary orientation $a$.
Then 
\begin{equation}
d_B^{T,a} (\calB(X), \calB(Y)) = d_I^{T,a} (X,Y).
\end{equation}
\end{theorem}
\begin{proof}
The statement follows from Corollary~\ref{cor:ISTdersub} and the derived equivalence \linebreak $\RHom(T(\Si),\blank) \colon \Db(\rep_{\Bbbk} A_n) \simto \Db(\rep_{\Bbbk} A_n(a))$ in Proposition~\ref{prp:tiltderived}.
%\qed
\end{proof}

The special case of Theorem~\ref{thm:ISTzigzag} is exactly an isometry theorem for purely zigzag persistence modules.

\section{Direct calculation of the induced distance on $\rep_{\Bbbk} A_n(a)$} \label{sec:Dircal}

We start this section with the following remark.

\begin{remark}
There are actually multiple derived equivalences from $\Db(\rep_{\Bbbk} A_n)$ to $\Db(\rep_{\Bbbk} A_n(a))$. 
For example, 
all translations of a classical tilting module are two-sided tilting complexes, which give derived equivalences (see \citealt{rickard1991derived}).
The induced distance on $\rep_{\Bbbk} A_n (a)$ from $\rep_{\Bbbk} A_n$ depends on the choice of derived equivalences. 
However, in all cases, the isometry theorem holds for $\rep_{\Bbbk} A_n (a)$ by Corollary~\ref{cor:ISTdersub}.
\end{remark}

The purpose of this section is to provide a direct calculation of the induced distance on $\rep_{\Bbbk} A_n(a)$.
Let us fix a classical tilting module $T:=T(\Si)$  given by a section $\Si$ such that $\Si^{\op}\cong A_n(a)$ (see Section~\ref{sec:main}).
Let $X$ be a representation of $A_n(a)$.
By Proposition~\ref{prp:tiltderived}, 
there exists a unique cochain complex $M^{\bullet}\in \Db (\rep_{\Bbbk} A_n)$ up to isomorphism such that $\RHom (T,M^{\bullet})\cong X$. 
This complex $M^{\bullet}$ is called \emph{the corresponding complex of $X$}.
It follows from Theorem~\ref{thm:ISTzigzag} that
the interleaving $d_I^{T,a}$ and bottleneck $d_B^{T,a}$ coincide, and thus we put $d^a:=d_I^{T,a}=d_B^{T,a}$.
By definition, for $X,Y\in \rep_{\Bbbk} A_n(a)$ we have 
\begin{equation}
d^a(X,Y)=d_I^D(M^{\bullet}, N^{\bullet}),
\end{equation}
where $M^{\bullet},N^{\bullet}$ are the corresponding complexes of $X,Y$, respectively (see Definition~\ref{dfn:dinterleave} and Definition~\ref{dfn:dereqint}).
Thus, to calculate $d^a$ on $\rep_{\Bbbk} A_n(a)$, we deal with the corresponding complexes.
For this purpose, we discuss the classical tilting torsion theory (for details, see \citealt[Chapter VI]{assem2006elements}).

Let $\calT$ be the full subcategory of $\rep_{\Bbbk} A_n$ consisting of representations $V$ satisfying $\Ext^1 (T,V)=0$ and $\calF$ the full subcategory of $\rep_{\Bbbk} A_n$ consisting of representations $V$ satisfying $\Hom (T,V) =0$.
Moreover, let $\calX$ be the full subcategory of $\rep_{\Bbbk} A_n (a)$ consisting of representations $V$ satisfying $T \otimes V = 0$ and $\calY$ the full subcategory of $\rep_{\Bbbk} A_n (a)$ consisting of representations $V$ satisfying $\Tor_1 (T,V) = 0$.
It is known that $\calT,\calF,\calX,$ and $\calY$ are closed under extensions.
Namely, for a short exact sequence 
\begin{equation}
0 \to M \to N \to L \to 0, 
\end{equation}
$M,L\in \calT$ (resp.\ $\calF, \calX$, or $\calY$) implies $N\in\calT$ resp.\ $\calF, \calX$, or $\calY$).
\begin{remark}
The pairs $(\calT,\calF)$ and $(\calX,\calY)$ are the so-called \emph{torsion pairs} (for definition, see \citealt[Chapter VI.1, 1.1 Definition]{assem2006elements}).
Moreover, these pairs are \emph{splitting}. 
Namely, for each indecomposable representation $M\in \rep_{\Bbbk} A_n$, 
we have $M\in \calT$ or $M\in \calF$,
and for each indecomposable representation $N\in \rep_{\Bbbk} A_n(a)$, 
we have $N\in \calX$ or $N\in \calY$ (see also \citealt[Chapter VI.5]{assem2006elements}).
\end{remark}

The functor $\Hom(T,\blank)$ gives an equivalence from $\calT$ to $\calY$ and the functor $\Ext^1 (T,\blank)$ gives an equivalence from $\calF$ to $\calX$ (see \citealt[Chapter VI.3, 3.8 Theorem]{assem2006elements}).
Using the right derived functor $\RHom(T,\blank)$, we can restate these results as follows. 

\begin{proposition} \label{prp:tilting}
The derived equivalence $\RHom(T,\blank)$ induces equivalences from $\calT[0]$ to $\calY[0]$ and from $\calF[1]$ to $\calX[0]$.
\end{proposition}

\begin{proof}
Note that we have $H^0 (\RHom(T,\blank))\cong \Hom (T,\blank)$ and 
\begin{equation}H^1 (\RHom(T,\blank))\cong
H^0 (\RHom(T,\blank[1]) \cong \Ext^1 (T,\blank).\end{equation}
Thus, the claim follows from \citet[Chapter VI.3, 3.8 Theorem]{assem2006elements}.
%\qed
\end{proof}

Proposition~\ref{prp:tilting} states that for a given indecomposable representation $X\in\rep_{\Bbbk} A_n(a)$, 
the corresponding complex $M^{\bullet}$ is a stalk complex $M^{\bullet}=L[0]$ or $L[1]$ for some $L\in \calT$ or $\calF$, respectively.  
The representation $L$ is called the \emph{corresponding representation} of $X$.

Furthermore, with the AR quiver, Proposition~\ref{prp:tilting} can be pictured as in Figure~\ref{fig:tiltingtorsion}, where $F=\Hom(T.\blank), G=\Ext^1(T,\blank),$ and $\RF=\RHom(T,\blank)$. 
\begin{figure}[h]
    \centering
    \includegraphics[width=10cm,height=8cm]{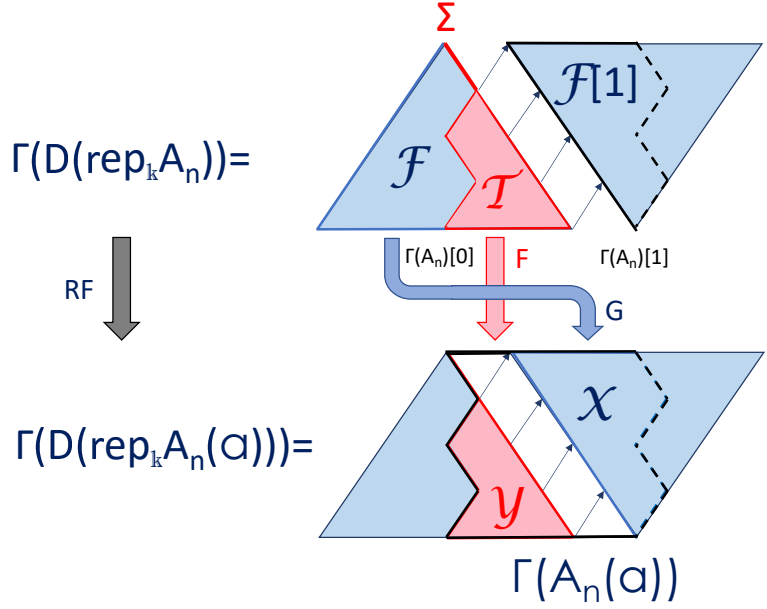}%{torsion.png}
    \caption{Correspondence between the AR quivers of $\rep_{\Bbbk} A_n$ and $\rep_{\Bbbk} A_n (a)$ in  the derived category}
    \label{fig:tiltingtorsion}
\end{figure}

Now we calculate the derived interleaving distance between two stalk complexes. 

\begin{proposition} \label{prp:calofderint}
Let $M,N$ be representations of $A_n$.
For each pair $i,j$ of integers,
\begin{equation}
d_I^D (M[i],N[j])=
\left\{
\begin{array}{cccc}
d_I (M,N), & i=j \\
\max \{d_I (M,0), d_I(N,0)\}, & i\not= j
\end{array}
\right.
.
\end{equation}
\end{proposition}
\begin{proof}
Note that 
\begin{equation}
H^l (M[i])=\left\{
\begin{array}{cc}
M,     & i=-l \\
0,     & i\not=-l
\end{array}
\right.,\quad
H^l (N[j])=\left\{
\begin{array}{cc}
N,     & j=-l \\
0,     & j\not=-l
\end{array}
\right.
.
\end{equation}
Thus, in the case that $i=j$, 
$M[i]$ and $N[j]$ are derived $\de$-interleaved if and only if $M$ and $N$ are $\de$-interleaved by Corollary~\ref{cor:dint}.
Otherwise, $M[i]$ and $N[j]$ are derived $\de$-interleaved if and only if 
$M$ and $0$ are $\de$-interleaved and $0$ and $N$ are $\de$-interleaved by Corollary~\ref{cor:dint}.
%\qed
\end{proof}

As a consequence, for indecomposable representations, we have the following calculation of the derived interleaving distance.

\begin{corollary} \label{cor:directcaldDI}
Let $M=\bbI[x,y],N=\bbI[s,t]$ be indecomposable representations of $A_n$.
For each pair $i,j$ of integers,
\begin{equation}
d_I^D (M[i],N[j])=
\left\{
\begin{array}{cc}
\min
\left\{
\begin{array}{c}
\max\{|x-s|,|y-t| \},   \\
\max\{ \lceil\frac{1}{2}  |y-x+1|\rceil,\lceil\frac{1}{2}|t-s+1|\rceil\}
\end{array}
\right\}, & i=j \\
\max \{\lceil\frac{1}{2} |y-x+1|\rceil,\lceil\frac{1}{2}|t-s+1|\rceil \}, & i\not= j
\end{array}
\right.
,
\end{equation}
where $\lceil \blank \rceil$ is the ceiling function.
\end{corollary}
\begin{proof}
Since 
\begin{equation}
d_I (M,N)=
\left\{
\begin{array}{cc}
\min\left\{
\begin{array}{c}
\max\{|x-s|,|y-t| \},  \\
\max\{ \lceil\frac{1}{2}  |y-x+1|\rceil,\lceil\frac{1}{2}|t-s+1|\rceil\} 
\end{array}
\right\}, & M\not =0, N\not=0 \\
\lceil\frac{1}{2} |y-x+1|\rceil, & N=0 \\
\lceil\frac{1}{2}|t-s+1|\rceil, \ & M=0
\end{array}
\right.,
\end{equation}
we obtain the desired statement by Proposition~\ref{prp:calofderint}.
%\qed
\end{proof}

By combining Proposition~\ref{prp:tilting} and  Proposition~\ref{prp:calofderint}, we can calculate the distance $d^a$ on $\rep_{\Bbbk} A_n(a)$ from the interleaving distance on $\rep_{\Bbbk} A_n$.

\begin{corollary} \label{cor:calofinddist}
Let $X,Y$ be indecomposable representations of $A_n (a)$ and $M, N$ the corresponding representations of $X,Y$ respectively.
Then 
\begin{equation}
d^a (X,Y)=
\left\{
\begin{array}{cc}
d_I(M,N),  & \text{ if } X,Y\in \calX \text{ or } X,Y\in \calY   \\
\max \{ d_I(M,0),d_I(N,0) \},  & \text{otherwise.} 
\end{array}
\right.
\end{equation}
\end{corollary}

By Corollary~\ref{cor:directcaldDI} and  Corollary~\ref{cor:calofinddist}, we can calculate the value $d^a(X,Y)$ concretely when we fix the orientation $a$.

\section{Comparison between the block distance and the induced distance}
\label{sec:comparison}

\citet{botnan2018algebraic}
proved an AST for purely zigzag persistence modules. 
In that paper, they introduced the interleaving and bottleneck distances on purely zigzag persistence modules.
\citet{bjerkevik2016stability} proved that those distances actually coincide.
Here, we refer to the interleaving distance as the \emph{block distance}, denoted by $d_{BL}$, following the paper \citet{meehan2020persistence}. 
 
In this section, based on the calculation in Section~\ref{sec:Dircal}, we will compare our induced distance $d^a$ with the distance $d_{BL}$.
\subsection{The block distance} \label{subsec:bldist}

First, we explain the block distance defined by \citet{botnan2018algebraic}.
For this aim, we will introduce the infinite purely zigzag quiver $\bbZ\bbZ$.

Let $\bbZ$ be the poset of integers with usual order and $\bbZ^{\op}$ its opposite poset.
As in \citet{botnan2018algebraic}, let $\bbZ \bbZ$ be the subposet of the poset $\bbZ^{\op} \times \bbZ$ given by 
\begin{equation}
\bbZ \bbZ:=\{(i,j)\mid i\in\bbZ, j\in\{i,i-1\} \}.
\end{equation}
Note that this can be expressed by the infinite purely zigzag quiver 
\begin{equation}
Q=
\vcenter{
\xymatrix{
& & (i+1,i+1) \ar@{.}[r]& \\
& (i,i) &  (i+1,i)  \ar[l] \ar[u]  \\ 
(i-1,i-1)  \ar@{.}[d]  & (i,i-1) \ar[l] \ar[u] \\
 &
}},
\end{equation}
so that $\bbZ \bbZ$ and $Q$ are identified and a (pointwise finite-dimensional) representation of $\bbZ \bbZ$ is just that of the quiver $Q$.
We use $\rep_{\Bbbk} \bbZ \bbZ$ to denote the category of representations of $\bbZ \bbZ$. 

Moreover, in \citet{botnan2018algebraic}, the intervals $\langle b,d \rangle_{\bbZ\bbZ}$ $(b\leq d)$ of $\bbZ\bbZ$ are classified into the following $4$ types:
\begin{equation}
\left\{
\begin{array}{lc}
\text{closed interval} & [b,d]_{\bbZ\bbZ}:=\{(i,j)\in \bbZ\bbZ \mid (b,b)\leq (i,j) \leq (d,d) \}, \\
\text{right-open interval} & [b,d)_{\bbZ\bbZ}:=\{(i,j)\in \bbZ\bbZ \mid (b,b)\leq (i,j) < (d,d) \}, \\
\text{left-open interval} &
(b,d]_{\bbZ\bbZ}:=\{(i,j)\in \bbZ\bbZ \mid (b,b) < (i,j) \leq (d,d) \}, \\
\text{open interval} & 
(b,d)_{\bbZ\bbZ}:=\{(i,j)\in \bbZ\bbZ \mid (b,b) < (i,j) < (d,d) \}. \\
\end{array}
\right.
\end{equation}

We use $\bbI^{\langle b,d \rangle_{\bbZ\bbZ}}$ to denote the interval representation of $\bbZ\bbZ$ associated with the interval $\langle b,d\rangle_{\bbZ\bbZ}$.
Note that the interval representation $\bbI^{\langle b,d\rangle_{\bbZ\bbZ}}$ of $\bbZ\bbZ$ is uniquely determined by the interval $\langle b,d\rangle_{\bbZ\bbZ}$.
Indeed, $\bbI^{\langle b,d \rangle_{\bbZ\bbZ}}$ is the representation given by 
\begin{equation}
\bbI^{\langle b,d \rangle_{\bbZ\bbZ}}_{(i,j)}
=
\left\{ \begin{array}{cc}
\Bbbk,  & (i,j) \in \langle b,d \rangle_{\bbZ\bbZ} \\
0,     & \text{otherwise}
\end{array}
\right.
\end{equation}
and is called a \emph{closed} (resp.\  \emph{right-open, left-open,} and \emph{open}) interval representation if $\langle b,d \rangle_{\bbZ\bbZ}$ is closed (resp.\ right-open, left-open, and open).
Note that interval representations are indecomposable and every pointwise finite representation of $\bbZ\bbZ$ can be decomposed into interval representations (see \citealt{botnan2017interval}).

The distance $d_{BL}$ is defined via the interleaving distance on 2D persistence modules.
For that purpose, \citet{botnan2018algebraic} defined an embedding functor $J \colon \rep_{\Bbbk} \bbZ\bbZ \to \rep_{\Bbbk} \bbU$, where $\bbU:=\{ (a,b) \in \bbR^2 \mid b \ge a \}$ and $\rep_{\Bbbk} \bbU$ is the category of representations of $\bbU$.
The functor $J$ was originally denoted by $E$ in loc.\ cit.
Note that $\bbU$ is a subposet of $\bbR^{\op} \times \bbR$ and that a poset can be expressed by a quiver $Q'$ with relations in general.
Hence, a representation of $\bbU$ is a representation of the quiver $Q'$ satisfying the condition induced from the relations.
For $\varepsilon \ge 0$, we set $\vec \varepsilon:=(-\varepsilon,\varepsilon) \in \bbR_{\ge 0}^2$ and define a shift functor $[\vec \varepsilon] \colon \rep_{\Bbbk} \bbU \to \rep_{\Bbbk} \bbU$ on objects by $M[\vec \varepsilon]_u:=M_{u+\vec \varepsilon}$ together with natural morphisms.
For $M \in \rep_{\Bbbk} \bbU$ and $\varepsilon \ge 0$, there is a canonical morphism $\phi_M^{\vec \varepsilon} \colon M \to M[\vec \varepsilon]$.

\begin{definition} \label{dfn:BLdist}
(1) For $M, N \in \rep_{\Bbbk} \bbU$ and $\varepsilon \ge 0$, $M$ and $N$ are said to be \emph{$\varepsilon$-interleaved} if there exist morphisms $f \colon M \to N[\vec \varepsilon]$ and $g \colon M \to N[\vec \varepsilon]$ such that 
the following diagrams commute:
\begin{equation}
\xymatrix{
M \ar[rr]^{\phi_M^{2\vec \varepsilon}} \ar[dr]_f && M[2\vec \varepsilon]\\
& N [\vec \varepsilon] \ar[ur]_{g[\vec \varepsilon]} &
},\quad
\xymatrix{
N \ar[rr]^{\phi_N^{2\vec \varepsilon}} \ar[dr]_g && N[2\vec \varepsilon]\\
& M [\vec \varepsilon] \ar[ur]_{f[\vec \varepsilon]} &
}.
\end{equation}
The \emph{interleaving distance on $\rep_{\Bbbk} \bbU$} is defined as 
\begin{equation}
    d_I^{\bbU}(M,N):=\inf \{ \varepsilon \ge 0 \mid \text{$M$ and $N$ are $\varepsilon$-interleaved}\}.
\end{equation}
(2) The \emph{block distance} $d_{BL}$ on $\rep_{\Bbbk} \bbZ\bbZ$ is defined as 
\begin{equation}
    d_{BL}(X,Y)
    := 
    d_I^{\bbU}(J(X),J(Y)).
\end{equation}
\end{definition}

The name of block distance will be justified below in this section.

We recall some properties of $J$, for which we need blocks and block-decomposable representations (see \citealt[Section 3]{botnan2018algebraic} or \citealt[Definition 2.5, Definition 2.10]{berkouk2019level}).

\begin{definition}
    A \emph{block} $B$ is a subset of $\bbR^2$ of the following type:
    \begin{enumerate}
        \item A \emph{birthblock} (\textbf{bb} for short) if there is $(a,b) \in \bbR^2$ such that $B=\langle -\infty,a \rangle \times \langle b,+\infty \rangle$, where $(a,b)$ can be $(+\infty,-\infty)$.
        Moreover, $B$ is said to be of type $\mathbf{bb}^+$ if $b>a$, and $\mathbf{bb}^-$ otherwise.
        \item A \emph{deathblock} (\textbf{db} for short) if there is $(a,b) \in \bbR^2$ such that $B=\langle a, +\infty \rangle \times \langle -\infty, b \rangle$.
        Moreover, $B$ is said to be of type $\mathbf{db}^+$ if $b>a$, and $\mathbf{db}^-$ otherwise.
        \item A \emph{horizontalblock} (\textbf{hb} for short) if there is $a \in \bbR$ and $b \in \bbR\cup\{+\infty\}$ such that $B=\bbR \times \langle a,b \rangle$.
        \item A \emph{verticalblock} (\textbf{vb} for short) if there is $a \in \bbR \cup \{ -\infty\}$ and $b \in \bbR$ such that $B=\langle a,b \rangle \times \bbR$.
    \end{enumerate}
\end{definition}

Each type of block is pictured in \pref{fig:block}.
Note that blocks are defined over the whole $\bbR^2$ not only $\bbU$.
Remark also that the type of blocks does not depend on whether its boundary is included or not.

\begin{figure}[h]
    \centering
    \begin{tikzpicture}
    \fill[cyan=!2] (0.5,1)--(0.5,-0.6)--(-1,-0.6)--(-1,1)--cycle;
    \draw[thick, blue=!30] (0.5,1)--(0.5,-0.6)--(-1,-0.6);
    \fill (0.5,0.5) circle [radius=2pt];
    \fill (-0.6,-0.6) circle [radius=2pt];
    \draw[thick] (0.5,0.5)--(-0.6,-0.6);
    \draw[->] (-1,0) -- (1,0);
    \draw[->] (0,-1) -- (0,1);
    \draw (-1,-1) -- (1,1);
    \node at (0,-1) [below] {$\mathbf{bb}^-$};
    
    \fill[pink=!5] (3-0.5,0.5)--(3+1,0.5)--(3+1,-1)--(3+-0.5,-1)--cycle;
    \draw[thick, red=!30, dashed] (3+-0.5,-1)--(3-0.5,0.5)--(3+1,0.5);
    \fill[yellow=!2] (3-0.5,0.5)--(3-0.5,1)--(3-1,1)--(3-1,0.5)--cycle;
    \draw[thick, orange=!30] (3-0.5,1)--(3-0.5,0.5)--(3-1,0.5);
    \fill (3-0.5,-0.5) circle [radius=2pt];
    \fill (3+0.5,0.5) circle [radius=2pt];
    \draw[thick] (3-0.5,-0.5)--(3+0.5,0.5);
    \draw[->] (3+-1,0)--(3+1,0);
    \draw[->] (3+0,-1)--(3+0,1);
    \draw (3+-1,-1)--(3+1,1);
    \node at (3+0,-1) [below] {$\mathbf{bb}^+$ and $\mathbf{db}^+$};
    
    \fill[lime=!5] (6-1,0.4)--(6+1,0.4)--(6+1,-0.2)--(6-1,-0.2)--cycle;
    \draw[thick, green=!30, dashed] (6-1,0.4)--(6+1,0.4);
    \draw[thick, green=!30] (6-1,-0.2)--(6+1,-0.2);
    \fill (6-0.2,-0.2) circle [radius=2pt];
    \fill (6+0.4,0.4) circle [radius=2pt];
    \draw[thick] (6-0.2,-0.2)--(6+0.4,0.4);
    \draw[->] (6+-1,0)--(6+1,0);
    \draw[->] (6+0,-1)--(6+0,1);
    \draw (6+-1,-1)--(6+1,1);
    \node at (6+0,-1) [below] {$\mathbf{hb}$};
    
    \fill[lime=!5] (9+0.7,1)--(9+0.7,-1)--(9+0.1,-1)--(9+0.1,1)--cycle;
    \draw[thick, green=!30] (9+0.7,1)--(9+0.7,-1);
    \draw[thick, green=!30, dashed] (9+0.1,-1)--(9+0.1,1);
    \fill (9+0.1,0.1) circle [radius=2pt];
    \fill (9+0.7,0.7) circle [radius=2pt];
    \draw[thick] (9+0.1,0.1)--(9+0.7,0.7);
    \draw[->] (9+-1,0)--(9+1,0);
    \draw[->] (9+0,-1)--(9+0,1);
    \draw (9+-1,-1)--(9+1,1);
    \node at (9+0,-1) [below] {$\mathbf{vb}$};
    
    \end{tikzpicture}
    \caption{In the first figure, a block of type $\mathbf{bb}^-$ pictured in blue.
    In the second, a block of type $\mathbf{db}^+$ is pictured in red and $\mathbf{bb}^+$ in yellow.
    In the third (resp.\ last) figure, a block of type $\mathbf{hb}$ (resp.\ $\mathbf{vb}$) is pictured in green.
    Remark that the type of blocks does not depend on whether its boundary lines are included or not.}
    \label{fig:block}
\end{figure}
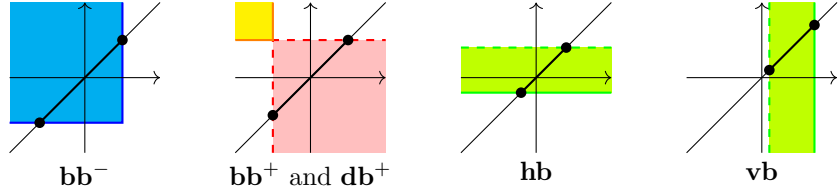

Using the above definition, we define block representations and block-decomposable representations of $\bbU$.
\begin{definition}
    \begin{enumerate}
        \item[(1)] A \emph{block representation $M$ of type $B$} of $\bbU$ is defined by, for $x\leq y \in \bbU$, 
        \begin{equation}
            M(x)=\left\{
            \begin{array}{cc}
            \Bbbk, & x\in \bbU \cap B  \\
            0,     & \text{otherwise}
            \end{array}
            \right.
            \text{ and }
            M(x \leq y)=\left\{
            \begin{array}{cc}
            \id_\Bbbk, & x,y\in \bbU \cap B  \\
            0,     & \text{otherwise}
            \end{array}.
            \right.
        \end{equation}
        Note that any block representation is indecomposable.
        \item[(2)] A representation $M$ of $\bbU$ is called \emph{block-decomposable} if $M$ can be only decomposed into block representations.  
    \end{enumerate}
\end{definition}

Remark that a block representation can be $0$ when the corresponding block is of type $\mathbf{db}^-$.

The functor $J$ sends an object of $\rep_{\Bbbk} \bbZ\bbZ$ to a block-decomposable persistence module.
In fact, each type of interval module is sent as follows.

\begin{lemma}[{\citealt[Lemma 4.1]{botnan2018algebraic}}]\label{lem:send-block}
    The functor $J$ sends closed, open, right-open, and left-open interval representations to block representations of type $\mathbf{bb}^-, \mathbf{db}^+, \mathbf{hb}$, and $\mathbf{vb}$, respectively.
\end{lemma}

By \pref{lem:send-block} and \citet[Lemma 3.1]{botnan2018algebraic}, we have the following calculation of $d_{BL}$ for the $4$ kinds of representations above.

\begin{proposition} \label{prp:BL}
Let $\langle b,d \rangle_{\bbZ\bbZ},\langle e,f\rangle_{\bbZ\bbZ}$ be intervals of $\bbZ\bbZ$.
Then the following holds.
\begin{equation}
d_{BL} (\bbI^{\langle b,d\rangle_{\bbZ\bbZ}},0)=
\left\{
\begin{array}{cl}
\infty,     & \langle b,d \rangle_{\bbZ\bbZ} \text{ is closed}
\\
\frac{1}{2}|d-b|, & \langle b,d \rangle_{\bbZ\bbZ} 
\text{ is half-open}
\\
\frac{1}{4}|d-b|, & \langle b,d \rangle_{\bbZ\bbZ}
\text{ is open}
\end{array}.
\right.
\end{equation}
Moreover,  
if $\langle b,d \rangle_{\bbZ\bbZ},\langle e,f \rangle_{\bbZ\bbZ}$ have the same type, then
\begin{equation}
d_{BL} (\bbI^{\langle b,d \rangle_{\bbZ\bbZ}},\bbI^{\langle e,f \rangle_{\bbZ\bbZ}})=
\min\left\{
\begin{array}{c}
\max\{|b-e|,|d-f|\}, \\
\max\{d_{BL} (\bbI^{\langle b,d \rangle_{\bbZ\bbZ}},0),d_{BL} (\bbI^{\langle e,f \rangle_{\bbZ\bbZ}},0)  \} 
\end{array}
\right\}.
\end{equation}
Otherwise, 
\begin{equation}
d_{BL} (\bbI^{\langle b,d \rangle_{\bbZ\bbZ}},\bbI^{\langle e,f \rangle_{\bbZ\bbZ}})=
\max\{d_{BL} (\bbI^{\langle b,d \rangle_{\bbZ\bbZ}},0),d_{BL} (\bbI^{\langle e,f \rangle_{\bbZ\bbZ}},0)  \}.
\end{equation}
\end{proposition}

\subsection{Comparison in purely zigzag setting} \label{subsec:comparison}

In this subsection, we will directly compare our induced distance $d^a$ with the block distance $d_{BL}$ in the purely zigzag setting.
Now we consider the quiver $A_n(z_1)$ where $n$ is odd. 
In this case, we denote the induced distance by $d^{z_1}$ instead of $d^a$ (see Section~\ref{sec:Dircal}).

For this aim we consider an injection $\mu_1 \colon A_n(z_1)\to \bbZ\bbZ$ defined by
\begin{equation}
\mu_1(x)=
\left\{
\begin{array}{ll}
(i,i), & x=2i-1 \\
(i+1,i), & x=2i 
\end{array}.
\right.
\end{equation}
The injection $\mu_1$ induces a fully faithful functor $\widetilde{\mu_1} \colon \rep_{\Bbbk} A_n (z_1) \to \rep_{\Bbbk} \bbZ \bbZ$ defined by
\begin{equation}
\widetilde{\mu_1}(X)(i,j)=
\left\{
\begin{array}{cc}
X(\mu_1(x)), & (i,j)=\mu_1(x) \\
0,     & \text{otherwise}
\end{array}
\right.
\end{equation}
for each representation $X$ of $A_n (z_1)$.
Then the functor $\widetilde{\mu_1}$ sends interval representations of $A_n(z_1)$ to those of $\bbZ  \bbZ$.
More precisely, we have the following result.
\begin{proposition} \label{prp:division}
    For any interval representation $\bbI[s,t]$ of $A_n(z_1)$, 
    \begin{enumerate}
        \item $\bbI[s,t]\in \calY$ if and only if $\widetilde{\mu_1}(\bbI[s,t])$ is closed or right-open, 
        \item $\bbI[s,t]\in \calX$ if and only if $\widetilde{\mu_1}(\bbI[s,t])$ is open or left-open.
    \end{enumerate}
\end{proposition}

To prove this proposition, we need the following lemma.

\begin{lemma} \label{lem:division}
For any interval representation $\bbI[s,t]$ of $A_n(z_1)$,
\begin{enumerate}
    \item $\bbI[s,t]\in \calY$ if $s$ is odd, 
    \item $\bbI[s,t]\in \calX$  otherwise. 
\end{enumerate}
\end{lemma}
\begin{proof}
For any pair $1\leq s \leq t \leq n$, 
there is an indecomposable stalk complex $Y_{s,t}$ in $\calT[0]$ or in $\calF [1]$  of $A_n$ such that $\RHom (T,Y_{s,t}) \cong \bbI[s,t]$ by Proposition~\ref{prp:tilting}.
Let $Q(s)=\Hom (T,X_s)\in \calY$ be the indecomposable projective representation of $A_n(z_1)$ corresponding to the vertex $1\leq s\leq n$. 
Note that $Q(s) \cong \bbI[s,s]$ if $s$ is odd and $Q(s)\cong \bbI[s-1,s+1]$ otherwise.
Then we have $\Hom (X_1,Y_{1,t}) \cong \Hom (Q(1),\bbI[1,t])\not=0$ for any $t$.
Since $X_1=P(1)$ is projective-injective and a direct summand of $T$, 
$\Hom (X_1,Y_{1,t}) \not = 0$
implies that $Y_{1,t}\in \calT[0]$ and hence $\bbI[1,t]\in \calY$ by Proposition~\ref{prp:tilting}.
For any odd integer $s> 1$, we have an exact sequence
\begin{equation}
0 \to \bbI[s,s] \to \bbI[s,t] \oplus \bbI[1,s] \to \bbI[1,t] \to 0.
\end{equation}
Since $\calY$ is closed under extensions, 
$\bbI[s,s],\bbI[1,t]\in \calY$ implies $\bbI[s,t]\in \calY$.

Statement (b) then follows from (a) and the fact that the torsion pair $(\calX,\calY)$ is splitting. 
%\qed
\end{proof}

\begin{proof}[{Proof of Proposition~\ref{prp:division}}]
$\bbI[s,t]\in \calY$ if and only if $s$ is odd.
In this case, $\mu_1(s)=(m,m)$ with $s=2m-1$.
Then, $s$ is odd if and only if  $\widetilde{\mu_1}(\bbI[s,t])$ is closed or right-open, as desired.%\\
%\qed
\end{proof}

\begin{remark}
We have the same result of Proposition~\ref{prp:division} in the case that $n$ is even. 
Moreover, we can also consider an injection $\mu_2 \colon A_n(z_2)\to \bbZ\bbZ$ defined by 
\begin{equation}
\mu_2(x)=
\left\{
\begin{array}{ll}
(i,i), & x=2i \\
(i+1,i), & x=2i-1 
\end{array}.
\right.
\end{equation}
Then we obtain the similar results of Proposition~\ref{prp:division} in the case of $A_n(z_2)$.  

Thus, Proposition~\ref{prp:division} and the similar results tell us that 
the AR quiver of $A_n(z_l)$ $(l=1,2)$ can be divided into $2$ areas consisting of $4$ kinds of intervals in the sense of \citet{botnan2018algebraic} with respect to classical tilting torsion theory.
\end{remark}

By Proposition~\ref{prp:division}, the interval representations $\bbI[s,t]$ of $A_n(z_1)$ can be classified into $4$ kinds of representations $\bbI^{\langle b,d \rangle_{\bbZ\bbZ}}$.
More precisely, we have the following correspondence between $(s,t)$ and $(b,d)$:
\begin{equation}
\left\{
\begin{array}{lcl}
\text{closed interval} & [b,d]_{\bbZ\bbZ} & (s=2b-1,t=2d-1),\\
\text{right-open interval} & [b,d)_{\bbZ\bbZ} & (s=2b-1,t=2d-2),\\
\text{left-open interval} &(b,d]_{\bbZ\bbZ} & (s=2b,t=2d-1),\\
\text{open interval} & (b,d)_{\bbZ\bbZ} & (s=2b,t=2d-2).
\end{array}
\right.
\end{equation}
Since $1\leq s \leq t \leq n$, we have $1\leq b \leq d \leq \lceil\frac{n}{2}\rceil$.
In this setting, by Proposition~\ref{prp:division}, $\bbI[s,t]\in \calY$ if and only if $\widetilde{\mu_1}(\bbI[s,t])$ is closed or right-open.
We use $\calY_c,\calY_{co}$ to denote the sets of interval representations $\bbI[s,t]\in \calY$ which correspond to closed or right-open interval representations of $\bbZ\bbZ$, respectively.
Similarly, we use $\calX_o,\calX_{oc}$ to denote the sets of interval representations $\bbI[s,t]\in \calX$ which correspond to open or left-open interval representations of $\bbZ\bbZ$, respectively.

From the proof of Proposition~\ref{prp:division}, we recall that  
$s$ is odd if and only if $\widetilde{\mu_1}(\bbI[s,t])$ is closed or right-open, and that
$t$ is odd if and only if $\widetilde{\mu_1}(\bbI[s,t])$ is closed or left-open. 
Let $I$ be an interval representation of $A_n(z_1)$, and set \begin{equation}S_I:=\{s=1,\cdots,n \mid \Hom (\bbI[s,s], I)\not = 0 \text{ or } \Hom (I, \bbI[s,s])\not = 0 \}.\end{equation}
In this case, $\bbI[s,s]$ is simple projective if $s$ is odd, and is simple injective otherwise.  
Consequently, $\Hom (I, \bbI[s,s])= 0$ if $s$ is odd, and $\Hom (\bbI[s,s], I) = 0$ otherwise.
Thus, we have $I=\bbI[s,t]$ with $s:=\min S_I$ and $t:=\max S_I$.
Note that simple projective representations are source vertices and simple injective representations are sink vertices in the AR quiver (see \citealt[Chapter IV.3, 3.6 Corollary]{assem2006elements}).
Then $\calY_c,\calY_{co}, \calX_{o},$ and $\calX_{oc}$ can be expressed in the AR quiver $\AR(A_n(z_1))$ of $A_n(z_1)$ as in Figure~\ref{fig:fig5}.

\begin{figure}[h!]
    \centering
    \includegraphics[width=7.5cm,height=5cm]{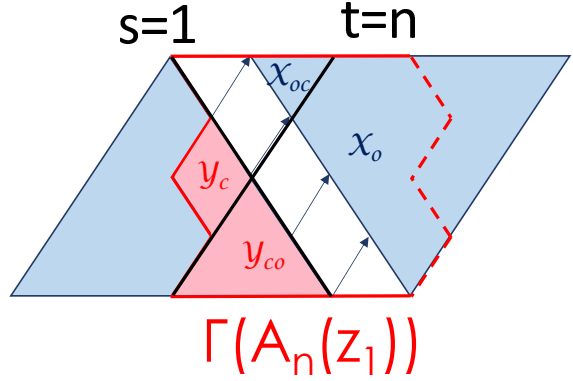}
    \caption{Division of the AR quiver $\AR(A_n(z_1))$ of $A_n(z_1)$}
    \label{fig:fig5}
\end{figure}

It is noteworthy that \citet{meehan2020persistence} give the same division of the AR quiver of purely zigzag persistence modules as our model.

The division of the AR quiver gives us the following correspondence between the interval representation of $A_n$ and $\bbZ\bbZ$.

\begin{lemma} \label{lem:corresponding}
Let $\bbI[s,t]$ be an interval representation of $A_n(z_1)$.
Then for the interval representation $\bbI^{\langle b,d\rangle_{\bbZ\bbZ}}:=\widetilde{\mu_1}(\bbI[s,t])$ of $\bbZ\bbZ$,
we have the corresponding representation $\bbI[x,y]\in \rep_{\Bbbk} A_n$ of $\bbI[s,t]$,  
where $(x,y)$ is given by the following:
\begin{equation}
\left\{
\begin{array}{lcl}
(x,y)=(b,n-d+1) & \langle b,d\rangle_{\bbZ\bbZ}=[b,d]_{\bbZ\bbZ} & (s=2b-1,t=2d-1),\\
(x,y)= (b,d-1) & \langle b,d\rangle_{\bbZ\bbZ} = [b,d)_{\bbZ\bbZ} & (s=2b-1,t=2d-2),\\
(x,y)=(n-d+2,n-b+1) & \langle b,d\rangle_{\bbZ\bbZ}=(b,d]_{\bbZ\bbZ} & (s=2b,t=2d-1),\\
(x,y)=(d,n-b+1) & \langle b,d\rangle_{\bbZ\bbZ}=(b,d)_{\bbZ\bbZ} & (s=2b,t=2d-2).
\end{array}
\right.
\end{equation}
\end{lemma}
\begin{proof}
The endpoint formulas can be easily calculated by Proposition~\ref{prp:tilting} and the above argument.
%\qed
\end{proof}

Proposition~\ref{prp:division} and Proposition~\ref{prp:BL} lead to the following.

\begin{proposition} \label{prp:comparison}
Let $\bbI[s,t],\bbI[u,v]$ be interval representations of $A_n(z_1)$.
For the interval representations $\bbI^{\langle b,d \rangle_{\bbZ\bbZ}}:=\widetilde{\mu_1}(\bbI[s,t]),\bbI^{\langle e,f \rangle_{\bbZ\bbZ}}:=\widetilde{\mu_1}(\bbI[u,v])$ of $\bbZ\bbZ$, the following inequalities hold:
\begin{enumerate}
    \item[$(1)$] $ d_{BL} (\bbI^{\langle b,d \rangle_{\bbZ\bbZ}},0) \leq d^{z_1}(\bbI[s,t],0) $ if 
    $\bbI[s,t]\in\calY_{co},\calX_{o},\calX_{oc}$,
    \item[$(2)$] $ d_\bbI^{z_1}(\bbI[s,t],0) < d_{BL} (\bbI^{\langle b,d \rangle_{\bbZ\bbZ}},0)=\infty,$
    if $\bbI[s,t]\in\calY_{c}$,
    \item[$(3)$] $ d_{BL}(\bbI^{\langle b,d \rangle_{\bbZ\bbZ}}, \bbI^{\langle e,f \rangle_{\bbZ\bbZ}}) \leq d^{z_1}(\bbI[s,t],\bbI[u,v])$
    if $\bbI[s,t],\bbI[u,v] \in \calY_{co}$, $\calX_o$ or $\calX_{oc}$, 
    or if $\bbI[s,t]\in \calY_{co}$, $\bbI[u,v]\in \calX$, and
    \item[$(4)$] $ d^{z_1}(\bbI[s,t],\bbI[u,v]) \leq d_{BL}(\bbI^{\langle b,d \rangle_{\bbZ\bbZ}}, \bbI^{\langle e,f \rangle_{\bbZ\bbZ}})$
    if $\bbI[s,t]\in \calY_c$.
\end{enumerate}
\end{proposition}

\begin{proof}
In each case, the value of $d_{BL}$ can be calculated by Proposition~\ref{prp:BL}.
On the other hand, the value of $d^{z_1}$ in each case can be calculated by Lemma~\ref{lem:corresponding}, Corollary~\ref{cor:directcaldDI}, and Corollary~\ref{cor:calofinddist}. 

Let us first compute $d^{z_1}(\bbI[s,t],0)$.
We have
\begin{equation}
d^{z_1}(\bbI[s,t],0)=
\left\{
\begin{array}{cc}
\lceil \frac{1}{2}|d-b| \rceil,  &  \bbI[s,t]\in \calY_{co} \text{ or } \calX_{oc} \\ 
\lceil \frac{1}{2}|n-(b+d)+2| \rceil, & \bbI[s,t]\in \calY_c \text{ or } \calX_o  
\end{array}
\right..
\end{equation}
When $1\leq b\leq d \leq \lceil\frac{n}{2}\rceil$, it is easy to check the inequality $|d-b| < |n-(b+d)+2| $. Thus, we obtain inequality (1).
Since $d^{z_1}(\bbI[s,t],0)< \infty$ always holds,
we obtain inequality (2).

Moreover, by the symmetry of the distance,  
$d:=d^{z_1}(\bbI[s,t],\bbI[u,v])$ can be calculated as follows. 
\begin{enumerate}
    \item[(a)] if $\bbI[s,t],\bbI[u,v]\in \calY_c$, then \begin{equation}
    d=\min\{\max\{ |b-e|, |d-f| \},\max\{\lceil \frac{1}{2}|n-(b+d)+2| \rceil, \lceil \frac{1}{2}|n-(e+f)+2| \rceil \}\},
    \end{equation}   
    \item[(b)] if $\bbI[s,t],\bbI[u,v]\in \calY_{co}$ or  $\bbI[s,t],\bbI[u,v]\in \calX_{oc}$, then 
    \begin{equation}
    d=\min\{\max\{ |b-e|, |d-f| \},\max\{\lceil \frac{1}{2}|d-b| \rceil, \lceil \frac{1}{2}|f-e| \rceil \}\},
    \end{equation} 
    \item[(c)] if $\bbI[s,t],\bbI[u,v]\in \calX_o$, then 
    \begin{equation}
    d=\min\{\max\{ |b-e|, |d-f| \},\max\{\lceil \frac{1}{2}|n-(b+d)+2| \rceil, \lceil \frac{1}{2}|n-(e+f)+2| \rceil \}\},
    \end{equation} 
    \item[(d)] if $\bbI[s,t]\in\calY_c, \bbI[u,v]\in \calY_{co}$, then 
    \begin{equation}
    d=\min\{\max\{ |b-e|, |n-(d+f)+2| \},\max\{\lceil \frac{1}{2}|n-(b+d)+2| \rceil, \lceil \frac{1}{2}|f-e| \rceil \}\},
    \end{equation}   
    \item[(e)] if $\bbI[s,t]\in\calX_{oc},\bbI[u,v]\in\calX_o$, then 
    \begin{equation}
    d=\min\{\max\{ |b-e|, |n-(d+f)+2| \}, 
    \max\{\lceil \frac{1}{2}|d-b| \rceil, \lceil \frac{1}{2}|n-(e+f)+2| \rceil \}\},\end{equation}
    and
    \item[(f)] if $\bbI[s,t] \in \calY, \bbI[u,v]\in \calX$, then
    \begin{equation}
    d=\max\{d_I^a(\bbI[s,t],0),d_I^a(\bbI[u,v],0) \}.\end{equation} 
\end{enumerate}

Since the inequality $|b-d|<|n-(b+d)+2|$
$(1\leq b\leq d\leq \lceil\frac{n}{2}\rceil)$ holds, by (1), the inequality 
\begin{equation}
d_{BL} (\bbI^{\langle b,d \rangle_{\bbZ \bbZ}},\bbI^{\langle e,f \rangle_{\bbZ \bbZ}}) \leq d^{z_1}(\bbI[s,t],\bbI[u,v]) 
\end{equation}
holds in cases (b), (c), and (f) except for when $\bbI[s,t]\in \calY_c$ in case (f).
Thus, we obtain the inequality (3).

In case (f), if $\bbI[s,t]\in \calY_c$, then $d_{BL}(\bbI^{[b,d]_{\bbZ\bbZ}},0)=\infty,$ 
and hence 
\begin{equation}d^{z_1}(\bbI[s,t],\bbI[u,v]) < d_{BL}(\bbI^{[b,d]_{\bbZ\bbZ}},\bbI^{[e,f]_{\bbZ\bbZ}}).\end{equation}
In case (a), since $\bbI[s,t],\bbI[u,v]\in \calY_c$, 
$d_{BL} (\bbI^{[b,d]_{\bbZ\bbZ}},0)=d_{BL}(\bbI^{[e,f]_{\bbZ\bbZ}},0)=\infty$. 
Then by definition, the inequality
\begin{equation} d^{z_1}(\bbI[s,t],\bbI[u,v])\leq d_{BL}(\bbI^{[b,d]_{\bbZ\bbZ}},\bbI^{[e,f]_{\bbZ\bbZ}})=\max\{|b-e|,|d-f|\} \end{equation} holds.
In case (d), since $\bbI[s,t]\in\calY_c$, $\bbI[u,v]\in \calY_{co}$, 
\begin{equation} d_{BL}(\bbI^{[b,d]_{\bbZ\bbZ}},\bbI^{[e,f)_{\bbZ\bbZ}})= \max\{d_{BL}(\bbI^{[b,d]_{\bbZ\bbZ}},0),d_{BL}(\bbI^{[e,f)_{\bbZ\bbZ}},0) \}=\infty.\end{equation} 
Then it is obvious that $d^{z_1}(\bbI[s,t],\bbI[u,v]) < d_{BL}(\bbI^{[b,d]_{\bbZ\bbZ}},\bbI^{[e,f]_{\bbZ\bbZ}})$.
Thus, we obtain the inequality (4).
%\qed
\end{proof}

\begin{remark}
In Proposition~\ref{prp:comparison}, the case in which  $\bbI[s,t]\in\calX_{oc}$, $\bbI[u,v]\in\calX_o$ remains.
In this case, we have 
\begin{equation}
\begin{array}{rcl}
d_{BL}(\bbI^{(b,d]_{\bbZ\bbZ}},\bbI^{(e,f)_{\bbZ\bbZ}})
&=&\max\{d_{BL}(\bbI^{(b,d]_{\bbZ\bbZ}},0),d_{BL}(\bbI^{(e,f)_{\bbZ\bbZ}},0)\} \\
&=&\max\{\frac{1}{2}|d-b|,\frac{1}{4}|f-e| \}, 
\end{array}
\end{equation}
and $d_{BL}$ and $d^{z_1}$ are incomparable for large $n$.
For example, 
in the case that $n=7$, we consider representations $\bbI[2,7]$, $\bbI[2,6]$, and $\bbI[1,2]$ of $A_n(z_1)$.
Then we have $\widetilde{\mu_1}(\bbI[2,7])=\bbI^{(1,4]_{\bbZ\bbZ}}$, $\widetilde{\mu_1}(\bbI[2,6])=\bbI^{(1,4)_{\bbZ\bbZ}}$, and $\widetilde{\mu_1}(\bbI[2,2])=\bbI^{(1,2)_{\bbZ\bbZ}}$.
By Proposition~\ref{prp:BL}, Lemma~\ref{lem:corresponding}, Corollary~\ref{cor:directcaldDI}, and Corollary~\ref{cor:calofinddist}, the inequalities 
\begin{equation}
d_{BL}(\bbI^{(1,4]_{\bbZ\bbZ}},\bbI^{(1,4)_{\bbZ\bbZ}})=\frac{3}{2} > 1= d^{z_1}(\bbI[2,7],\bbI[2,6])\end{equation}
and
\begin{equation}
d_{BL}(\bbI^{(1,4]_{\bbZ\bbZ}},\bbI^{(1,2)_{\bbZ\bbZ}})=\frac{3}{2} < 3=d^{z_1}(\bbI[2,7],\bbI[2,2])
\end{equation}
hold, as desired.
\end{remark}

We conclude that the block distance $d_{BL}$ in \citet{botnan2018algebraic} and our induced distance $d^{z_1}$ are incomparable. 
Indeed, Proposition~\ref{prp:comparison} (3) and (4) tell us that the inequality for comparing the distances is dependent on interval type.

\subsection{Comparison in arbitrary zigzag setting} \label{subsec:arbzigzag}
For any zigzag persistence module $M \in \rep_{\Bbbk} A_n(a)$, 
we can associate a purely zigzag persistence module $\widetilde{M}\in \rep_{\Bbbk} A_m(z_1)$ with $m\geq n$ by adding identity morphisms. 
We may assume that $m$ is the minimum with respect to adding identity morphisms. 
Namely, the number of added identity morphisms can be minimized as much as possible.
For example, $A_n(a)$ is assumed to be a quiver of the form $1 \to 2 \to 3 \leftarrow 4$. Then a representation $M$ of the quiver is of the form
\begin{equation}
M= M_1 \to M_2 \to M_3 \leftarrow M_4 
\end{equation}
and $\widetilde{M}$ is defined by 
\begin{equation}
\widetilde{M} = M_1 \xleftarrow{\id} M_1 \to M_2 \xleftarrow{\id} M_2 \to M_3 \leftarrow M_4 \in \rep_{\Bbbk} A_6 (z_1).
\end{equation}

The correspondence $M \mapsto \widetilde{M}$ induces a fully faithful functor 
\[\iota_a \colon \rep_{\Bbbk} A_n(a) \to \rep_{\Bbbk} A_m(z_1).\]
Similarly, we obtain a fully faithful functor $\iota'_a \colon \rep_{\Bbbk} A_n(a) \to \rep_{\Bbbk} A_m(z_2)$.
Through the functor $\iota_a$ (resp.\ $\iota'_a$), the category $\rep_{\Bbbk} A_n(a)$ is regarded as a full subcategory of $\rep_{\Bbbk} A_m(z_1)$ (resp.\ $\rep_{\Bbbk} A_m(z_2)$).
Thus, by an isometry theorem for purely zigzag persistence modules in \citet{botnan2018algebraic,bjerkevik2016stability}, we obtain an isometry theorem for arbitrary zigzag persistence modules. 
The induced block distance on arbitrary zigzag persistence modules is denoted by $d_{BL}^a$.
Namely, for $M,N \in \rep_\Bbbk A_n(a)$,
\begin{equation}
    d_{BL}^a(M,N):=d_{BL}(\widetilde{\mu_1}(\iota_a (M)), \widetilde{\mu_1}(\iota_a (N))).
\end{equation}

In this case,  our induced distance $d^a$ (see Section~\ref{sec:Dircal}) on zigzag persistence modules is also incompatible with the distance $d_{BL}^a$, as we will see below.
We may assume that $n\geq 4$ since there are no non-purely zigzag persistence modules in the case that $n\leq 3$.
Moreover, we may assume that $A_n(a)$ is not purely zigzag.
Then $m\geq 5$.

\noindent (1) We will confirm $d^a < d_{BL}^a$ for a closed interval representation.
There is at least one odd integer $t$ such that $\bbI[1,t]=\iota_a(\bbI [1, c])$ for some $1\leq c \leq n$
since the morphism $\widetilde{M}_{m-1} \to \widetilde{M}_m$ is not an added identity morphism by the minimality of $m$.
Since $\bbI[1,t] \in \calY_{c}$ and $d^a$ is always finite, the inequality 
\begin{equation}d^a(\bbI[1,c],0) < d_{BL}^a (\bbI[1,c],0)=d_{BL}(\widetilde{\mu_1}(\bbI[1,t]),0)=\infty\end{equation}
holds by Proposition~\ref{prp:BL}.

\noindent (2) We will confirm $d^a < d_{BL}^a$ for a half-open interval representation.

First, if $1$ is a source vertex in $A_n(a)$, then 
we have $\bbI[1,2]=\iota_a(\bbI[1,1])\in \calY_{co}$.
In this case, $\widetilde{\mu_1}(\bbI[1,2])=I^{[1,2)_{\bbZ\bbZ}}$ is right-open. 
Then since the inequality $d^a(\bbI[1,1],0)\geq 1$ always holds,
the inequality 
\begin{equation}
d_{BL}^a(\bbI[1,1],0)=d_{BL}(\bbI^{[1,2)_{\bbZ\bbZ}},0)=\frac{1}{2} < d^a(\bbI[1,1],0)
\end{equation}
holds by Proposition~\ref{prp:BL}.

Second, in the case that the quiver $A_n(a)$ is of the form $1\leftarrow 2 \to 3 \leftrightarrow \cdots \leftrightarrow n$,
the inequality $d_{BL}^a(\bbI[1,2],0)< d^a(\bbI[1,2],0)$ holds similarly as above.

Third, if the quiver $A_n(a)$ is of the form $1\leftarrow 2 \leftarrow 3 \leftrightarrow \cdots \leftrightarrow n$, 
then $\bbI[2,3] =\iota_a(\bbI[2,2]) \in \calX_{oc}$.
In this case, $\widetilde{\mu_1}(\bbI[2,3])=I^{(1,2]_{\bbZ\bbZ}}$ is left-open.
Then since the inequality $d^a(\bbI[2,2],0)\geq 1$ always holds,
the inequality 
\begin{equation}
d_{BL}^a(\bbI[2,2],0)=d_{BL}(\bbI^{(1,2]_{\bbZ\bbZ}},0)=\frac{1}{2} < d^a(\bbI[2,2],0)
\end{equation}
holds by Proposition~\ref{prp:BL}. 

In cases (1) and (2), the inequality for comparing the distances $d^a$ and $d_{BL}^a$ is dependent on interval type, hence they are incomparable.

\section{Relation to sheaf-theoretic distance} \label{sec:compsheaf}

In this section, we study the relation to the convolution distance introduced by \citet{kashiwara2018persistent}.
The convolution distance is defined as a distance on the derived category $\Db(\Sh_c(\Bbbk_\bbR))$ of constructible sheaves on $\bbR$ (indeed in a more general setting).
In \pref{subsec:quiver-and-sheaf}, we prove a subcategory of $\Sh_c(\Bbbk_\bbR)$ is equivalent to the category $\rep_{\Bbbk} \bbZ\bbZ^{\op}$ of representations of the opposite quiver of the infinite purely zigzag quiver $\bbZ \bbZ$ (see \pref{subsec:bldist}).
Through the equivalence the convolution distance induces a distance on $\Db(\rep_{\Bbbk} \bbZ\bbZ^{\op})$.
In \pref{subsec:BL-convolution}, we show that $0$-th cohomology part of the induced distance coincides with the block distance $d_{BL}$.
In \pref{subsec:sheaf-equioriented}, we also see that the derived interleaving distance on ordinary persistence modules can be realized as a modified version of the convolution distance.

\subsection{Quiver representations and constructible sheaves}\label{subsec:quiver-and-sheaf}

In this subsection, we prove the category of representations of the infinite zigzag quiver $\bbZ\bbZ^{\op}$ is equivalent to some sheaf category on $\bbR$.
The equivalence induces a distance on the derived category of such representations.

First, let us briefly recall the notion of sheaves and fix some notation.
Let $X$ be a topological space and $\operatorname{Open}(X)$ the category of open subsets of $X$ whose Hom-set $\Hom(U,V)$ is the singleton if $U \subset V$ and empty otherwise.
A sheaf $F$ of $\Bbbk$-vector spaces on $X$ is a functor $\operatorname{Open}(X)^{\op} \to \Vect(\Bbbk)$ with some gluing condition (see \citealt{KS90} for example).
We write $\Sh(\Bbbk_X)$ for the abelian category of sheaves of $\Bbbk$-vector spaces on $X$.
In what follows, we focus on sheaves on $\bbR$.
A sheaf $F \in \Sh(\Bbbk_\bbR)$ is said to be \emph{constructible} if there exist  discrete points $\{x_k\}_{k \in \bbZ}$ with $x_k < x_{k+1} $ such that $F|_{(x_k, x_{k+1})}$ is locally constant for any $k \in \bbZ$ and $F_t$ is finite-dimensional for any $t \in \bbR$.
We denote by $\Sh_c(\Bbbk_\bbR)$ the full subcategory of $\Sh (\Bbbk_{\bbR})$ consisting of constructible sheaves.

\begin{definition}
    One defines $\Sh_{\bbZ}(\Bbbk_\bbR)$ as the full subcategory of $\Sh_c(\Bbbk_\bbR)$ consisting of objects $F$ such that $F|_{(i,i+1)}$ is constant for any $i \in \bbZ$.
\end{definition}

Now we consider the relation between representations of the infinite zigzag quiver and sheaves on $\bbR$, which is essentially studied by \citet{Gu16}.
For $X \in \rep_{\Bbbk} \bbZ\bbZ^{\op}$, we set
\begin{equation}
    F^X_t := 
    \begin{cases}
        X_{(i,i)}, & t=i \in \bbZ \\
        X_{(i,i-1)}, & i-1 \le t \le i \ (i \in \bbZ)
    \end{cases}
\end{equation}
and define $S(X) \in \Sh_{\bbZ}(\Bbbk_\bbR)$ by
\begin{equation}
    S(X)(U):=
    \left\{ 
        f \in \prod_{t \in U}F^X_t 
        \; \middle| \; 
        \begin{aligned}
            & \text{$f|_{U \setminus \bbZ}$ is locally constant, } \\
            & \text{for any $y \in \bbZ \cap U$ and $\varepsilon>0$ small enough}, \\  
            & f(y-\varepsilon)=X_{\be_y}(f(y)) \text{ with $\be_y = \alpha_{(y,y-1),(y,y)}$}, \\ 
            & f(y+\varepsilon)=X_{\be'_{y}}(f(y)) \text{ with $\be'_y = \alpha_{(y+1,y),(y,y)}$} 
        \end{aligned}
    \right\}.
\end{equation}
Note that with a morphism $\varphi \colon X \to Y$ one can associate a canonical morphism $S(\varphi) \colon S(X) \to S(Y)$.
The correspondence defines a functor $S \colon \rep_{\Bbbk} \bbZ\bbZ^{\op} \to \Sh_{\bbZ}(\Bbbk_\bbR)$.
For this functor, we have the following equivalence.

\begin{proposition}\label{prp:rep-sheaf-equivalence}
    The functor $S \colon \rep_{\Bbbk} \bbZ\bbZ^{\op} \to \Sh_{\bbZ}(\Bbbk_\bbR)$ is an equivalence of categories.
\end{proposition}

\begin{proof}
    We define a functor $R \colon \Sh_{\bbZ}(\Bbbk_\bbR) \to \rep_{\Bbbk} \bbZ\bbZ^{\op}$ as follows.
    Let $F \in \Sh_{\bbZ}(\Bbbk_\bbR)$.
    For $i \in \bbZ$, we set  
    \begin{equation}
        X_{(i,i)}:= F((i-1,i+1)) \cong F_i, 
        \quad 
        X_{(i,i-1)} := F((i-1,i)).
    \end{equation}
    We also define morphisms $X_{(i,i)} \to X_{(i,i-1)}$ as $F((i-1,i+1)) \to F((i-1,i))$ and $X_{(i-1,i-1)} \to X_{(i,i-1)}$ as $F((i-2,i)) \to F((i-1,i))$ for $i \in \bbZ$.
%     \begin{align}
%     X_{(i,i)} \to X_{(i,i-1)} & := F((i-1,i+1))  \to  F((i-1,i)), \\
%     X_{(i-1,i-1)} \to X_{(i,i-1)} & := F((i-2,i))  \to  F((i-1,i)). 
% %    \begin{cases}
% %    F((j-1,j+1)) \to F((j,j+1)), & i=2j+1 \text{ and } i'=2j+2\\
% %    F((j-1,j)) \leftarrow F((j-1,j+1)), & i=2j \text{ and } i'= 2j +1
% %    \end{cases}.
%     \end{align}
    Define $R(F):=X \in \rep_{\Bbbk} \bbZ\bbZ^{\op}$.
    One can show that $R$ is a quasi-inverse of $S$, which proves the result.
    %\qed
\end{proof}

\begin{remark}
    As seen in Subsections~\ref{subsec:comparison} and \ref{subsec:arbzigzag}, for any orientation $a$ we have the embedding $\widetilde{\mu_1}\circ\iota_a \colon \rep_\Bbbk A_n(a) \to \rep_\Bbbk \bbZ\bbZ$. 
    Combining this embedding with \pref{prp:rep-sheaf-equivalence}, we can regard arbitrary (not necessarily purely) zigzag persistence modules as constructible sheaves on $\bbR$.
    In particular, in the case $a=e$, ordinary persistence modules can also be regarded as constructible sheaves on $\bbR$, which we will study in \pref{subsec:sheaf-equioriented} in more detail.
\end{remark}

The convolution distance was introduced by \cite{kashiwara2018persistent} on the derived category $\Db(\Sh_c(\Bbbk_{\bbR^n}))$ of constructible sheaves on $\bbR^n$.
Here we recall it in the case $n=1$.
Let $q_1,q_2 \colon \bbR^2 \to \bbR$ be the first and second projections.
Moreover, we set $s \colon \bbR^2 \to \bbR, (t_1,t_2) \mapsto t_1+t_2$.
For $F^\bullet,G^\bullet \in \Db(\Sh_c(\Bbbk_{\bbR}))$, their convolution $F^\bullet \star G^\bullet \in \Db(\Sh_c(\Bbbk_{\bbR}))$ is defined by 
\begin{equation}
    F^\bullet \star G^\bullet := Rs_! (q_1^{-1}F \otimes q_2^{-1}G).
\end{equation}
For $\varepsilon \ge 0$, we set $K_\varepsilon:=\Bbbk_{[-\varepsilon,\varepsilon]} \in \Sh_c(\Bbbk_\bbR)$, which has stalks $\Bbbk$ on $[-\varepsilon,\varepsilon]$ and $0$ otherwise.
We have $K_{\varepsilon} \star K_{\varepsilon'} \cong K_{\varepsilon+\varepsilon'}$ for $\varepsilon, \varepsilon' \ge 0$.
For $\varepsilon \ge 0$, the canonical morphism $K_\varepsilon \to K_0$ induces a morphism $\phi_{F^\bullet}^\varepsilon \colon F^\bullet \star K_\varepsilon \to F^\bullet \star K_0 \cong F^\bullet$.
For $F^\bullet,G^\bullet \in \Db(\Sh_c(\Bbbk_{\bbR}))$ and $\varepsilon \ge 0$, $F^\bullet$ and $G^\bullet$ are said to be \emph{$\varepsilon$-isomorphic} if there exist morphisms $f \colon K_\varepsilon \star F^\bullet \to G^\bullet$ and $g \colon K_\varepsilon \star G^\bullet \to F^\bullet$ such that following diagrams commute:
\begin{equation}
\xymatrix{
F^\bullet \star K_{2\varepsilon} \ar[rr]^-{\phi_{F^\bullet}^{2\varepsilon}} \ar[dr]_-{f \star K_{\varepsilon}} && F^\bullet\\
& G^\bullet \star K_{\varepsilon} \ar[ur]_-{g} &
},\quad
\xymatrix{
G^\bullet \star K_{2\varepsilon} \ar[rr]^-{\phi_{G^\bullet}^{2\varepsilon}} \ar[dr]_-{g \star K_{\varepsilon}} && G^\bullet\\
& F^\bullet \star K_{\varepsilon} \ar[ur]_-{f} &
}.    
\end{equation}
The \emph{convolution distance} on $\Db(\Sh_c(\Bbbk_{\bbR}))$ is defined as 
\begin{equation}
    d_C(F^\bullet,G^\bullet)
    :=
    \inf\{ \varepsilon \ge 0 \mid \text{$F^\bullet$ and $G^\bullet$ are $\varepsilon$-isomorphic} \}.
\end{equation}

Through the equivalence in \pref{prp:rep-sheaf-equivalence}, the convolution distance induces a distance on $\Db(\rep_{\Bbbk} \bbZ\bbZ^{\op})$.
In the next subsection, we compare a non-derived version of the induced distance with the distances $d^{z_1}$ and $d_{BL}$ (see \pref{sec:Dircal} and \pref{dfn:BLdist}).

\subsection{Block distance and convolution distance}\label{subsec:BL-convolution}

In this subsection, we investigate the relation between the distance $d_{BL}$ and the convolution distance through the equivalence in \pref{prp:rep-sheaf-equivalence}.

\cite{berkouk2019level} considered a functor $\Xi \colon \rep_{\Bbbk} \bbU \to \Sh_c(\Bbbk_\bbR)^{\op}$ and \linebreak $\Psi \colon \Sh_c(\Bbbk_\bbR)^{\op} \to \rep_{\Bbbk} \bbU$, where $\bbU=\{ (a,b) \in \bbR^2 \mid b \ge a \}$.
Moreover, they proved that the functors commute with $[\vec \varepsilon]$ and $(\blank)\star K_\varepsilon$ in the derived setting, where $\vec \varepsilon=(-\varepsilon,\varepsilon)$.
Here we restate the result in a weaker form, i.e., in the non-derived setting.
For $F \in \Sh_c(\Bbbk_\bbR)$, we define 
\begin{align}
    F \star_{\text{nd}} K_\varepsilon & := H^0(F \star K_\varepsilon) \in \Sh_c(\Bbbk_\bbR), \\
    \phi_{F, \mathrm{nd}}^{\varepsilon} & :=H^0(\phi_{F}^{\varepsilon}) \colon F \starnd K_\varepsilon \to F \ \text{in $\Sh_c(\Bbbk_\bbR)$}.
\end{align}
By \citet[Proposition~3.22, Lemmas~3.27 and 3.28]{berkouk2019level} and \citet[Proposition~3.8]{berkouk2018derived}, for $M \in \rep_{\Bbbk} \bbU$ whose indecomposable summands are only blocks of type $\mathbf{bb}^-, \mathbf{db}^+, \mathbf{hb}$, and $\mathbf{vb}$, one has $\Xi(M[\vec \varepsilon]) \cong M \starnd K_\varepsilon$ for any $\varepsilon \ge 0$.
Similarly, for $F \in \Sh_c(\Bbbk_\bbR)$ one has $\Psi(F)[\vec \varepsilon] \cong \Psi(F \starnd K_\varepsilon)$ for any $\varepsilon \ge 0$ (cf.\ \citealt[Proposition~4.16]{berkouk2019level}).
Moreover, they satisfy $\Xi \circ \Psi \simeq \id$.

Now we consider the relation to the equivalence in \pref{prp:rep-sheaf-equivalence}.
We define $\Theta$ as the composite $\rep_{\Bbbk} \bbZ\bbZ \xrightarrow{D} (\rep_{\Bbbk} \bbZ\bbZ^{\op})^{\op} \xrightarrow{S^{\op}} \Sh_{\bbZ}(\Bbbk_\bbR)^{\op}$, where $D$ denotes the $\Bbbk$-dual functor.

\begin{proposition}\label{prp:commutativity}
    One has the following commutative diagram:
    \begin{equation}
        \xymatrix{
            \rep_{\Bbbk} \bbZ\bbZ \ar[r]^-{J} \ar[rd]_-{\Theta} & \rep_{\Bbbk} \bbU \ar[d]^-{\Xi} \\
            & \Sh_{\bbZ}(\Bbbk_\bbR)^{\op}
        }.
    \end{equation}
\end{proposition}

\begin{proof}
    Since any object of $\rep_{\Bbbk} \bbZ\bbZ$ is interval-decomposable, it suffices to show the commutativity for interval modules.
    On objects, we obtain the result from \citet[Lemma~4.1]{botnan2018algebraic} and \citet[Proposition~3.22]{berkouk2019level}.
    On morphisms, we only need to check non-zero morphisms between interval modules, which can be easily verified.
    %\qed
\end{proof}

We define a non-derived version of the convolution distance as follows.
For $F, G \in \Sh_c(\Bbbk_\bbR)$, $F$ and $G$ are said to be \emph{$H^0$-$\varepsilon$-isomorphic} if there exist $f \colon K_\varepsilon \starnd F \to G$ and $g \colon K_\varepsilon \starnd G \to F$ %satisfying 
such that the following diagrams commute:
\begin{equation}
\xymatrix{
F^\bullet \starnd K_{2\varepsilon} \ar[rr]^-{\phi_{F^\bullet}^{2\varepsilon}} \ar[dr]_-{f \starnd K_{\varepsilon}} && F^\bullet\\
& G^\bullet \starnd K_{\varepsilon} \ar[ur]_-{g} &
},
\xymatrix{
G^\bullet \starnd K_{2\varepsilon} \ar[rr]^-{\phi_{G^\bullet}^{2\varepsilon}} \ar[dr]_-{g \starnd K_{\varepsilon}} && G^\bullet\\
& F^\bullet \starnd K_{\varepsilon} \ar[ur]_-{f} &
}.    
\end{equation}
% \begin{enumerate}
%     \item the composite $F \starnd K_{2\varepsilon} \xrightarrow{f \starnd K_{\varepsilon}} G \starnd K_{\varepsilon} \xrightarrow{g} F$ is equal to $\phi_{F, \text{nd}}^{2\varepsilon}$, 
%     \item the composite $G \starnd K_{2\varepsilon} \xrightarrow{g \starnd K_{\varepsilon}} F \starnd K_{\varepsilon} \xrightarrow{g} G$ is equal to $\phi_{G, \text{nd}}^{2\varepsilon}$.
% \end{enumerate}
We define a non-derived convolution distance by 
\begin{equation}
    d_{C,\text{nd}}(F,G)
    := 
    \inf\{ \varepsilon \ge 0 \mid \text{$F$ and $G$ are $H^0$-$\varepsilon$-isomorphic} \}
\end{equation}
for $F,G \in \Sh_c(\Bbbk_\bbR)$.
It can be easily checked that for $F,G \in \Sh_c(\Bbbk_\bbR)$, the inequality 
\begin{equation}
    d_{C,\text{nd}}(F,G)\leq d_{C}(F,G).
\end{equation}
holds.
In particular, if both $F$ and $G$ have no indecomposable direct summand of the form $\Bbbk_{(a,b)}$ with open interval $(a,b)$ in $\bbR$, 
then the equality 
\begin{equation}
    d_{C,\text{nd}}(F,G)= d_{C}(F,G)
\end{equation}
holds (see \citealt[Proposition 4.1 and Proposition 4.2]{berkouk2018derived}).

Through the functor $\Xi$, the interleaving distance $d_I^{\bbU}$ on $\rep_\Bbbk \bbU$ is compatible with $d_{C,\text{nd}}$ on $\Sh_c(\Bbbk_\bbR)$ as follows. 
This is a non-derived version of a main result of \citet{berkouk2019level}.

\begin{lemma}[{cf.\ \citealt[Theorem~4.21]{berkouk2019level}}]\label{lem:almost-isometry}
    For any $M,N \in \rep_{\Bbbk} \bbU$ whose indecomposable summands are only blocks of type $\mathbf{bb}^-, \mathbf{db}^+, \mathbf{hb}$, and $\mathbf{vb}$, one has 
    \begin{equation}
        d_I^{\bbU}(M,N)=d_{C,\text{nd}}(\Xi(M),\Xi(N)).
    \end{equation}
\end{lemma}

\begin{proof}
    By the compatibility of $[\vec \varepsilon]$ and $(\blank) \starnd K_\varepsilon$, 
    for any $M,N \in \rep_{\Bbbk} \bbU$ whose indecomposable summands are only blocks of type $\mathbf{bb}^-, \mathbf{db}^+, \mathbf{hb}$, and $\mathbf{vb}$, we have 
    \begin{equation}
        d_{C,\mathrm{nd}}(\Xi(M),\Xi(N)) \le d_I^{\bbU}(M,N)
    \end{equation}
    and for any $F,G \in \Sh_c(\Bbbk_\bbR)$,
    \begin{equation}
        d_I^{\bbU}(\Psi(F),\Psi(G)) \le d_{C,\mathrm{nd}}(F,G).
    \end{equation}
    Similarly to \citet[Collorary~4.20]{berkouk2019level}, for any $M \in \rep_{\Bbbk} \bbU$ whose indecomposable summands are only blocks of type $\mathbf{bb}^-, \mathbf{db}^+, \mathbf{hb}$, and $\mathbf{vb}$, one has 
    \begin{equation}
        d_I^{\bbU}(M,\Psi\Xi(M))=0.
    \end{equation}
    Hence we obtain 
    \begin{align}
        d_I^{\bbU}(M,N) 
        & \le d_I^{\bbU}(M,\Psi\Xi(M))+d_I^{\bbU}(\Psi\Xi(M),\Psi\Xi(N))+d_I^{\bbU}(\Psi\Xi(N),N) \\
        & \le d_{C,\mathrm{nd}}(\Xi(M),\Xi(N)) \\
        & \le d_I^{\bbU}(M,N),
    \end{align}
    which proves the result.
    %\qed
\end{proof}

The following shows that the distance $d_{BL}$ coincides with the convolution distance on sheaves through the equivalence $\Theta$.

\begin{proposition} \label{prp:isoBLconv}
    For any $X,Y \in \rep_{\Bbbk} \bbZ\bbZ$, one has
    \begin{equation}
        d_{BL}(X,Y)
        = 
        d_{C,\text{nd}}(\Theta(X),\Theta(Y)).
    \end{equation}
\end{proposition}

\begin{proof}
    By Lemma~\ref{lem:send-block}, for $X \in \rep_{\Bbbk} \bbZ\bbZ$, $J(X)$ is a block-decomposable representation whose blocks are only of type $\mathbf{bb}^-, \mathbf{db}^+, \mathbf{hb}$, and $\mathbf{vb}$.
    Thus the result follows from \pref{prp:commutativity} and \pref{lem:almost-isometry}.
    %\qed
\end{proof}

As a consequence of Proposition~\ref{prp:comparison} and Proposition~\ref{prp:isoBLconv}, our induced distance is incomparable with the non-derived and the ordinary convolution distance defined above in the purely zigzag setting. 
Indeed, for $X,Y\in \calY_c$, we have
\begin{equation}
d^{z_1}(X,Y)\leq d_{BL}(X,Y)=        d_{C,\text{nd}}(\Theta(X),\Theta(Y))=        d_{C}(\Theta(X),\Theta(Y)).
\end{equation}
On the other hand, for $X,Y\in \calY_{oc}$, we have
\begin{equation}
d^{z_1}(X,Y)\geq d_{BL}(X,Y)=        d_{C,\text{nd}}(\Theta(X),\Theta(Y))=        d_{C}(\Theta(X),\Theta(Y)).
\end{equation}
Recall that $\calY_c$ and $\calY_{oc}$ are the set of interval representations in  $\rep_{\Bbbk} A_n(z_1)$ corresponds to closed and left-open interval representations of $\bbZ\bbZ$.

\subsection{Interleaving distance for equioriented quiver and convolution distance}\label{subsec:sheaf-equioriented}

In this subsection, we also see that the interleaving distance on $\Db(\rep_{\Bbbk} A_n)$ can be realized as a modified version of the convolution distance on $\Db(\Sh_c(\Bbbk_\bbR))$.
As seen in \pref{subsec:arbzigzag}, there exists a fully faithful functor \[\rep_{\Bbbk} A_m \xrightarrow{\iota_{e}} \rep_{\Bbbk} A_n(z_1) \xrightarrow{\widetilde{\mu_1}} \rep_{\Bbbk} \bbZ\bbZ.\]
We describe the corresponding subcategory of $\Sh_\bbZ(\Bbbk_\bbR)$ through the equivalence $\Theta \colon \rep_{\Bbbk} \bbZ\bbZ \xrightarrow{\sim} \Sh_\bbZ(\Bbbk_\bbR)^{\op}$.
For that purpose, we use the notion of the mircosupports of sheaves (see \citealt{KS90}).
For $F \in \Sh(\Bbbk_\bbR)$, one can define $\MS(F) \subset T^*\bbR$, which describes the singular co-direction of $F$.

\begin{definition}
    For $m \in \bbZ_{\ge 0}$, 
    $\Sh_{m,+}(\Bbbk_\bbR)$ denotes the full subcategory of $\Sh_\bbZ(\Bbbk_\bbR)$ consisting of objects $F$ satisfying the following conditions:
    \begin{enumerate}
        \item $\Supp(F) \subset [1,m+1]$,
        \item $\MS(F) \subset \bbR \times \bbR_{\ge 0}$.
    \end{enumerate}
\end{definition}

\begin{proposition}
    The restriction of the functor $\Theta$ induces an equivalence of categories $\Theta' \colon \rep_{\Bbbk} A_m \simto \Sh_{m,+}(\Bbbk_\bbR)^{\op}$.
\end{proposition}

\begin{proof}
    Let $M \in \rep_\Bbbk \bbU$ and set $F:=\Theta(M)$.
    By the definition of $\iota_{e}$ and $\Theta$, 
    $\Supp(F) \subset [1,m+1]$ and the restriction morphism $F((i-1,i+1)) \to F((i,i+1))$ is an isomorphism for any $i \in \bbZ$. 
    By \citet[Corollary~7.3]{Gu16}, $F$ is decomposed as 
    \begin{equation}
        F \cong \bigoplus_\alpha \Bbbk_{I_\alpha},
    \end{equation}
    where each $I_\alpha$ is a interval whose end points are in $\bbZ$.
    Combining this decomposition with the above isomorphisms, we find that each $I_\alpha$ is of the form $[a,b)$ for some $a<b \in \bbR$.
    Hence, we have $\MS(F) \subset \bbR \times \bbR_{\ge 0}$.
    %\qed
\end{proof}

We also denote by $\Theta' \colon \Db(\rep_{\Bbbk} A_m) \simto \Db(\Sh_{m,+}(\Bbbk_\bbR)^{\op})$ the induced derived equivalence.

For $\delta \in \bbZ_{\ge 0}$, we define a functor $(\delta) \colon \Db(\Sh_{m,+}(\Bbbk_\bbR)) \to \Db(\Sh_{m,+}(\Bbbk_\bbR))$ by $F^\bullet(\delta):=(F^\bullet \star K_\delta)_{[1,m+1)}$.
Then for $F^\bullet \in \Db(\Sh_{m,+}(\Bbbk_\bbR))$ there exists a canonical morphism $\phi_{F^\bullet,m,+}^\delta \colon F^\bullet(\delta) \to F^\bullet$ for any $\delta \in \bbZ_{\ge 0}$.
Using the morphism, one can define $\delta$-interleaving for objects of $\Db(\Sh_{m,+}(\Bbbk_\bbR))$ and a modified version of the convolution distance by
\begin{equation}
    d_{C,m,+}(F^\bullet,G^\bullet)
    := 
    \inf \{ \delta \in \bbZ_{\ge 0} \mid \text{$F^\bullet$ and $G^\bullet$ are $\delta$-interleaved} \}.
\end{equation}

The derived equivalence $\Theta'$ is a contravariant functor so that $\Theta'$ sends the diagram in Definition~\ref{dfn:dinterleave} to the opposite one. 
Thus we have the following. 

\begin{proposition} \label{prp:modifyconv}
    For $X^{\bullet} \in \Db(\rep_{\Bbbk} A_m)$, one has $\Theta'(X^{\bullet}(\delta)) \cong \Theta'(X^{\bullet})(\delta)$.
    In particular, 
    \begin{equation}
        d_I^D(X^{\bullet},Y^{\bullet}) 
        =
        d_{C,m,+}(\Theta'(X^{\bullet}),\Theta'(Y^{\bullet})).
    \end{equation}
\end{proposition}

\begin{remark}
In Subsection~\ref{subsec:arbzigzag}, we confirm that 
our induced distance $d^a$ and the block distance $d_{BL}$ are incomparable in the arbitrary zigzag setting.
In particular, in the equioriented setting, they are incomparable.
Thus, Proposition~\ref{prp:isoBLconv} and Proposition~\ref{prp:modifyconv} imply that 
the distances $d_{C}$ and $d_{C,m,+}$ on $\Db(\Sh_{m,+}(\Bbbk_\bbR))$ are incomparable.
\end{remark}

\section*{Acknowledgements}
%\begin{acknowledgements}
We are grateful to Kosuke Sakurai, who was a Master's student of the first author at Tohoku University, for his cooperation in %completing an isometry theorem from an AST in our derived setting. 
proving the converse AST in our derived setting. 
The second author thanks Nicolas Berkouk for helpful discussion.
The third author also thanks Emerson G.\ Escolar and Killian F.\ Meehan for their helpful discussions.

The first author was partially supported by JST CREST Mathematics (15656429) and JSPS Grant-in-Aid for Scientific Research (A) (JP20H00119).
The second author was partially supported by JST ACT-X (JPMJAX1903).
The third author was partially supported by JSPS Grant-in-Aid for Scientific Research (C)  (JP20K03760) and
Osaka City University Advanced Mathematical Institute (MEXT Joint Usage/Research Center on Mathematics and Theoretical Physics JPMXP0619217849).
%\end{acknowledgements}

\section*{Conflict of interest}
On behalf of all authors, the corresponding author states that there is no conflict of interest.

\bibliographystyle{spbasic}
\bibliography{ref}

\end{document}